\documentclass[11pt]{article}

\usepackage[numbers,sort&compress]{natbib}
\usepackage{enumerate}
\usepackage [latin1]{inputenc}
\usepackage{amscd}
\usepackage{amsmath}
\usepackage{latexsym}
\usepackage{amsfonts}
\usepackage{amssymb}
\usepackage{amsthm}
\usepackage{verbatim}
\usepackage{mathrsfs}
\usepackage{enumerate}
\usepackage{hyperref}

\theoremstyle{plain}\newtheorem{definition}{Definition}[section]
\theoremstyle{definition}\newtheorem{theorem}{Theorem}[section]
\theoremstyle{plain}\newtheorem{lemma}[theorem]{Lemma}
\theoremstyle{plain}\newtheorem{coro}[theorem]{Corollary}
\theoremstyle{plain}\newtheorem{prop}[theorem]{Proposition}
\theoremstyle{remark}\newtheorem{remark}{Remark}[section]
\usepackage{xcolor}

\newcommand{\fbxoop}{\int_{\Omega_{\varepsilon}}\!\,\!\!\!\!\!  \!\!\!\hspace{-0.08cm}-}

\newcommand{\norm}[1]{\left\|#1\right\|}
\newcommand{\Div}{\mathrm{div}\,}
\newcommand{\B}{\Big}

\newcommand{\be}{\begin{equation}}
\newcommand{\ee}{\end{equation}}
 \newcommand{\ba}{\begin{aligned}}
 \newcommand{\ea}{\end{aligned}}

\newcommand{\fbxoo}{\int_{B_{\varepsilon}(0)}\!\!\!\!\!\!\!\!\!\!\!\!\!\!\!\! -~\,~\,~\,~\,}

  \newcommand{\f}{\frac}
    
  \newcommand{\ben}{\begin{enumerate}}
   \newcommand{\een}{\end{enumerate}}

\newcommand{\Rmnum}[1]{\expandafter\@slowromancap\romannumeral #1@}

\allowdisplaybreaks

\numberwithin{equation}{section}
\begin{document}
%%%%%%%%%%%%%%%%%%%%%%%%%%%%%%%%%%%%%%%%%%%%%%%%%%%%%%%%%%%%%%%%%%%%%%%%%%%%%%%%%%%%%%%%%%%%%%%%%%%%
\title{A note on helicity conservation for compressible Euler equations in a bounded domain with vacuum}
\author{~ Yulin Ye\footnote{ School of Mathematics and Statistics,
		Henan University,
		Kaifeng, 475004,
		P. R. China. Email: ylye@vip.henu.edu.cn}     
    }
\date{}
\maketitle
\begin{abstract}
In this paper,   we consider the helicity conservation of weak solutions for the compressible Euler equations  in a bounded domain with general pressure law and   vacuum. We deduce a sufficient condition for a weak solution    conserving the helicity based on the interior Besov-VMO type regularity, the continuous conditions for  velocity and vorticity  near the boundary, and some regularities for density near vacuum.

  \end{abstract}
\section{Introduction}
\label{intro}
\setcounter{section}{1}\setcounter{equation}{0}
Let $\Omega$ be a connected domain in $\mathbb{R}^3$ with $C^2$ boundary $\partial\Omega$. In the present paper, we are concerned with the helicity conservation of weak solutions for the following compressible  Euler equations
\begin{equation}\left\{\ba\label{CEuler}
\rho_t+\nabla \cdot (\rho v)=0, \ \ \ \ \ \ \ \ \ \ \ \ \ \ \ \ \ \ \ \ \ \ \ \ \ \ \ \ \ \ \ \ \ \ &\text{in}\ \Omega\times(0,T),\\
(\rho v)_{t} +\Div(\rho v\otimes v)+\nabla
p(\rho )=0,\ \ \  \ \ \ \ \ \ \ \ \ \ \ &\text{in}\ \Omega\times(0,T),\\
u(t,x)\cdot\overrightarrow{n}(x)=0,\ \ \ \ \ \ \ \ \ \ \ \ \ \ \ \ \ \ \ \ \ \ \ \ \ \ \ \ \ \ \ \  \ &\text{on}\ \partial\Omega\times(0,T),\\
\ea\right.\end{equation}
where  the scalar function $\rho$ and the unknown vector $v $ represent   the density and  velocity of the flow,  respectively; $p(\rho)$ stands for the general  pressure of the fluid, which satisfies $p(\rho)\in C([0,\infty))\cap C^2(0,\infty)$; $\overrightarrow{n}(x) $ denotes the outward normal vector field to the boundary $\partial \Omega$. It is noted that when $p(\rho)=\rho^\gamma $ with $\gamma >1$, then  system \eqref{CEuler}  is usually called isentropic compressible Euler equations.
To solve equations \eqref{CEuler}, the initial data
is  complemented   by
\begin{equation}\label{Euler1}
	\rho(0,x)=\rho_0(x),\ (\rho v)(0,x)=(\rho_0 v_0)(x),\ x\in \Omega,
\end{equation}
where we define $v_0 =0$ on the sets $\{x\in \Omega:\ \rho_0=0\}.$

Formally, when $\rho=constant$, then the compressible Euler equations \eqref{CEuler} reduce to the ideal incompressible Euler equations
\begin{equation}\left\{\begin{aligned}\label{Euler}
		&v_{t}+v\cdot \nabla v+\nabla \pi=0,~~\text{in}~~~~\Omega\times(0,\,T),\\ &\text{div}\, v=0,~~~~~~~~~~~~~~~~~~\text{in}~~~~\Omega\times(0,\,T),\\
		&v\cdot\overrightarrow{n}=0,~~~~~~~~~~~~~~~~~\text{on}~~~\partial\Omega\times(0,\,T),\\
		&v|_{t=0}=v_{0}(x)~~~~~~~~~~~~~~\text{on} ~~~ \Omega\times\{t=0\}.
	\end{aligned}\right.\end{equation} It is well-known that there exist two physically conserved  quantities for the regular solutions of the Euler system \eqref{Euler}:
 \begin{itemize}
\item  energy  conservation
\be\label{ec}
\f12\int_{\Omega}|v(t,x)|^{2}dx=\f12\int_{\Omega}|v(x,0)|^{2}dx;
\ee
\item helicity preservation
\be\label{hc}
\int_{\Omega} \omega(t,x)\cdot v(t,x)  dx=\int_{\Omega} \omega(x,0)\cdot v(x,0)  dx,
\ee
 \end{itemize}
where $\omega=curl v$ is the vorticity of the fluid.
The energy of fluid is
 the
core  in  Kologromv theory in 1941 and the Onsager conjecture in 1949.
   Onsager in \cite{[Onsager]} conjectured  that weak solutions of the Euler equation with
  H\"older continuity exponent $\alpha>\frac{1}{3}$ do conserve energy and that turbulent
  or anomalous dissipation occurs when $\alpha\leq \frac{1}{3}$.  For the recent advances in both two directions of the Onsager's conjecture, the readers can be referred to  \cite{[Eyink],[CET],[DR],[CCFS],[Shvydkoy2009],[Shvydkoy2010],[FW2018],[DS0],[DS1],[DS2],[DS3],[Isett]} and \cite{[WZ],[Yu],[Zhou]} for other related models. In \cite{[NNT2020]}, Nguyen, Nguyen and Tang deduced  the energy conservation criterion for weak solutions of the equations \eqref{CEuler} as follows:
\be\ba\notag
&0\leq \rho \in L^{\infty} (0, T;L^\infty(\mathbb{T}^{d} )) \cap L^{\infty} (0, T ; {B}_{3,c(\mathbb{N})}^{\f13} (\mathbb{T}^{d} ) ), v\in  L^{3} (0, T ; {B}_{3,c(\mathbb{N})}^{\f13} (\mathbb{T}^{d} ) ),\\
& {\bf 1}_{\rho\leq\delta_0}\nabla \rho\in L^\infty (0,T;\mathbb{T}^d),{\bf 1}_{\rho\leq\delta_0}\partial_t\rho\in L^{3} (0, T;\mathbb{T}^{d} )\ \text{and}\ \lim\limits_{\eta\to 0}| \{(x,t)\in Q_T, \rho \leq \eta\}|=0,\\
 \ea\ee
 where
  $\delta>0$. Moreover, in \cite{[NNT2020]}, for the bounded domain case, the following condition to handle the boundary layer is also required
  \be\ba\label{NNTboundarycondition}
&\B(\int_{0}^{T}\fbxoop~~ |v|^{3}dxdt\B)^{\f23}\B(\int_{0}^{T}\fbxoop~~ |v\cdot \vec{n}(x)|^{3}dxdt\B)^{\f13}=o(1), \ \text{as} \ \varepsilon\to 0,\\
&\int_{0}^{T}\fbxoop~~ |v\cdot \vec{n}(x)|^{3}dxdt=o(1), \ \text{as} \ \varepsilon\to 0,\ea\ee
where $\Omega_{\varepsilon}=\{x\in\Omega|d(x,\partial\Omega)\leq\varepsilon\}$.
 Subsequently,
   Swierczewska, Gwiazda, Titi and Wiedemann \cite{[BGSTW]}  showed that
   the spaces $B^{\alpha}_{3,\infty}(\Omega), \alpha>1/3$  to  $\underline{B}^{\f13}_{3,VMO}( \Omega)$
together with the boundary condition
 \be\label{1.8}
  \lim_{h\rightarrow0}\int_{0}^{T}\int_{\{x\in\Omega|\f{h}{2}<d(x,\partial\Omega)<h\}}\!\!\!\!\!\!\!\!\!\!\!\!\!\!\!\!\!\!\!\!\!\!\!\!\!\!\!\!\! \!\!\!\!\!\!\!\!\!\!\!\!\!\!\!\!\!\!\!\!\hspace{0.03cm}-~~~~~~~~~~~~~~~~~~~~~~
  \B|\B[\f{|m|^{2}}{2\rho}  +  \Pi (\rho )+\pi(\rho)\B]\f{m}{\rho}\cdot \vec{n}(\sigma(x))\B|dxdt=0\ee
 ensure that   the energy  is locally conserved in the sense of distributions.

The description of turbulence not only rests on the cascades of energy, but also the helicity. The interaction of the transfer of energy and helicity plays a critical  role in determining the direction of their  cascade.    History, the concept of the
 helicity  in an inviscid fluid was proposed  by Moffatt in \cite{[Moffatt]}. As pointed in \cite{[Moffatt],[MT]},
helicity is important at a fundamental level in relation to flow kinematics
because it admits topological interpretation in relation to the linkage or
linkages of vortex lines of the flow.
Indeed,
Moffatt \cite{[Moffatt]} examined that the    helicity of smooth solution of
the compressible Euler equations  \eqref{CEuler} is invariant in time. Regarding to helicity conservation of weak solutions  for homogeneous incompressible Euler equations,  Chae in \cite{[Chae]} considered the preservation of the helicity  of the weak solutions for the Euler equations  via $\omega$
in spaces $ L^{3}(0,T;B^{\alpha}_{\f95,\infty}) $ with $\alpha>\f13$.
Subsequently,
it is shown that $v \in L^{r_{1}}(0,T;B^{\alpha}_{\f92,q})$ and
$\omega \in  L^{r_{2}}(0,T;\in B^{\alpha}_{\f95,q})$ with $\alpha>\f13$, $q\in [2,\infty]$, $r_1\in [2,\infty]$, $r_2\in [1,\infty]$ and $\frac{2}{r_{1}}+\frac{1}{r_{2}}=1$  ensure  the helicity conservation of the weak solutions by Chae in \cite{[Chae1]}.  After that,  Cheskidov, Constantin, Friedlander and Shvydkoy \cite{[CCFS]} established the helicity conservation class based on the velocity $v$ in  Onsager's critical space
$   L^{3}(0,T;B^{\f23}_{3,c(\mathbb{N})}).$  De Rosa \cite{[De Rosa]} proved that the helicity is a constant provided
$v\in L^{2r}(0,T;W^{\theta, 2p})$ and  $ \omega\in L^{\kappa}(0,T;W^{\alpha, q})$ with  $\f{1}{p}+\f1q=\f{1}{r}+\f{1}{\kappa}$  and $2\theta+\alpha\geq1$.
 Very recently, the dissipation term resulted from the
lack of smoothness of the solutions for the helicity balance in \eqref{Euler} was recently deduced by Boutros  and  Titi in \cite{[BT]}.
The global helicity balance to the Euler equations \eqref{Euler} on bounded domains
in terms of the boundary contributions of the vorticity, velocity and pressure
under suitable regularity assumptions were presented by Inversi and M. Sorella \cite{[IS]}.
 However, there   has been few literature  concerning the preserving  helicity  of weak solutions for
the    compressible Euler equations \eqref{CEuler} allowing vacuum or on a bounded domain.

 The key difficulty in compressible case is that one can not get the equation of vorticity $\omega$ from the momentum equation directly like the incompressible case, due to the low regularity of $(\rho,v)$, the strong coupling of density $\rho$ and velocity $v$ and possible vacuum. Recently, the authors in \cite{[WWY]} observed that for any smooth solutions $(\rho, v)$ of compressible Euler equations \eqref{CEuler} with $0<c_1\leq \rho\leq c_2<\infty$, dividing the both sides of the momentum equation $\eqref{CEuler}_2$ by $\rho$, then we can reformulate the system    \eqref{CEuler} as
\be\left\{\ba\label{rCEuler1}
&\rho_t+\nabla \cdot (\rho v)=0, \\
& v_{t} + v\cdot\nabla v+\nabla
P(\rho )=0,
\ea\right.\ee
where  $P(\rho )= \int_{1}^{\rho}\f{p'(s)}{s}ds$.
 Then the momentum equation  in \eqref{rCEuler1} is the same as that for
the incompressible case.  However, this is not valid for weak solutions $(\rho,v)$ even with $0<c_1\leq \rho\leq c_2<\infty$. To do this, the author first proved that  the weak solution of compressible Euler equations \eqref{CEuler} is also the weak solution of \eqref{rCEuler1} in the sense of distributions by commutators techniques  under $0<c_1\leq \rho\leq c_2<\infty$, and then proved the helicity conservation for compressible Euler equations in the periodic domain $\mathbb{T}^d$ without vacuum. It should be noted that the proof in \cite{[WWY]}  strongly relies on the density away from vacuum and no boundary effect. Inspired by works \cite{[NNT2020],[IS]}, the main objective of this paper is  to give a sufficient condition on the helicity conservation of weak solutions to the compressible Euler equations \eqref{Euler} in a bounded domain and the density allowing vacuum.
We formulate our main result on the helicity conservation of weak solutions for the  compressible Euler equations \eqref{CEuler} as follows.
\begin{theorem}\label{the1}
	Let the pair $(\rho, v)$ be a weak solution of the compressible Euler equations  \eqref{CEuler} in the Definition of \ref{ceulerdefi} and $\omega\in C([0,T];L^{\f{3}{2}}(\Omega))$, $v\in C([0,T];L^{3}(\Omega))$.  Assume that	\begin{enumerate}[(i)]
		\item (Interior Besov regularity)  if $(\rho,v)$ satisfy
		\begin{equation}\label{a1}
			\ba
			&	0\leq \rho \leq c<\infty,~\rho \in L^3(0,T;B^{\f{1}{3}}_{3,c(\mathbb{N})}(\Omega))\cap  L^\infty(0,T;B^{\frac{1}{3}}_{\infty,\infty}(\Omega)), \\
					&v \in L^3(0,T;\underline{B}^{\frac{1}{3}}_{3,VMO}(\Omega))\  \text{and}~  \omega\in L^3(0,T;L^3(\Omega)) .\ea\end{equation}
		\item (Continuous condition near the boundary )
		\begin{equation}\label{a2}
		\lim\limits_{\varepsilon\to 0}\int_0^T {\int\!\!\!\!\!\!-~}_{\f{\varepsilon}{2}<d(x,\partial \Omega)<\varepsilon}	\B|\left(P(\rho )-\f12|v|^{2}\right)\omega\cdot  \vec{n}(\sigma(x))+\left(\omega\cdot v\right)v\cdot \vec{n}(\sigma(x))\B|dxdt
	=  0,
		\end{equation}
	where $P(\rho )= \int_{1}^{\rho}\f{p'(s)}{s}ds$.
		\item (Near the vacuum)for any $\delta>0$, if
		\begin{equation}\label{a4}
			\rho^{-1}{\bf 1}_{\rho \leq \delta}\in L^{r_1}(0,T;L^{r_1}(\Omega)), \partial_t \rho{\bf 1}_{\rho \leq \delta}\in L^{r_2}(0,T;L^{r_2}(\Omega)), ~\text{and}~ \nabla \rho{\bf 1}_{\rho \leq \delta}\in L^{r_3}(0,T;L^{r_3}(\Omega)),
		\end{equation}
	\end{enumerate}
	for some positive constant $c,r_1,r_2,r_3$ and some sufficiently small constant $\delta>0$ such that $\frac{1}{r_1}+\frac{1}{r_2}=\frac{2}{3}$ and $\frac{1}{r_1}+\frac{1}{r_3}=\frac{1}{3}$, then the helicity of compressible Euler equation is conserved, that is for any $0<t<T$,
	$$	\int_{\Omega} \omega(t,x)\cdot v(t,x)  dx=\int_{\Omega} \omega(x,0)\cdot v(x,0)  dx.$$
	
\end{theorem}
To prove this theorem, we apply the  test function used by Nguyen, Nguyen and Tang in \cite{[NNT2020]} to capture the effect of the boundary and allow vacuum.
A slight modification the proof of this theorem together with the test function in  \cite{[BGSTW]}
yields that
\begin{coro}\label{coro1}
	Let $v$ be a weak solution of the incompressible Euler equations \eqref{Euler} and $\omega\in C([0,T];L^{\f{3}{2}}(\Omega))$, $v\in C([0,T];L^{3}(\Omega))$.  Suppose that $v$ and $\omega$ satisfy one of the following three conditions
 \begin{enumerate}[(1)]
  \item
    $            v\in L^{3} (0,T; \underline{B}^{\f23}_{3,VMO}(\Omega));$
 \item $ v\in L^{p_{1}} (0,T;\underline{B}^{\alpha}_{q_{1},VMO}  (\Omega)), \omega\in L^{ p_{2}  } (0,T; B^{\beta}_{q_{2},\infty}(\Omega)),$ with $\f{2}{p_1}+\f{1}{p_2}=1,\f{2}{q_1}+\f{1}{q_2}=1,2\alpha+\beta\geq1;$
\item $ v\in L^{p_{1}} (0,T;B^{\alpha}_{q_{1},\infty} (\Omega)), \omega\in L^{ p_{2}  } (0,T; \underline{B}^{\beta}_{q_{2},VMO} (\Omega)),$ with $\f{2}{p_1}+\f{1}{p_2}=1,\f{2}{q_1}+\f{1}{q_2}=1,2\alpha+\beta\geq1.$
      \end{enumerate}
Meanwhile, there holds
\be\ba\label{HCboundary}
&\lim_{\varepsilon\rightarrow0}\int_{0}^{T}\fbxoo
  \B|  [\Pi
 \omega+\f12|v |^{2}]\cdot \omega\cdot n(\sigma(x))  \B|dxdt\\&+  \lim_{\varepsilon\rightarrow0}\int_{0}^{T}\fbxoo
  \B| (v\times\omega)\cdot (  v\times n(\sigma(x)))  \B|dxdt=0.
      \ea\ee
      Then the helicity is constant, that is 
      $$\f{d}{dt}\int_{\Omega}v\cdot\omega dx=0$$
      holds in the sense of distributions.
	\end{coro}
The rest of this paper is organized as follows.
 In Section 2, we present some notations and  auxiliary lemmas which will be used in the present paper.  In the spirit of \cite{[Yu]},  we  will  show 
   the      Constantin-E-Titi type commutator estimates   in a bounded domain  in terms of
   the  functions  in Besov-VMO type spaces $\underline{B}^{\f13}_{p,VMO}$, which plays an important role in the  following study.
 Section 3 is devoted to  the proof of helicity conservation of weak solutions for the compressible   Euler equations in a bounded domain.
\section{Notations and some auxiliary lemmas} \label{section2}

First, we introduce some notations used in this paper. For a bounded domain $\Omega\subset \mathbb{R}^d$, we denote $ dist(x,\partial\Omega)$ as the distance function for any $x\in \Omega $, and define the following subdomains as follows
$$\Omega_{\varepsilon}=\{x\in \Omega| dist(x,\partial\Omega)< \varepsilon\}, \ \Omega^\varepsilon=\Omega\setminus \overline{\Omega}_\varepsilon,$$
where $\overline{\Omega}_\varepsilon$ represents the closure of $\Omega_\varepsilon$ in the Euclidean topology. For any Borel set $E\subset \mathbb{R}^d$, we denote ${\int_{E}\!\!\!\!\!\!\!\! -~}=\f{1}{\mathcal {L}(E)}\int _E f dx$
.

Next, we let $\varrho:\mathbb{R}^{d}\rightarrow \mathbb{R}$ be a standard mollifier.i.e. $\varrho(x)=C_0e^{-\frac{1}{1-|x|^2}}$ for $|x|<1$ and $\varrho(x)=0$ for $|x|\geq 1$, where $C_0$ is a constant such that $\int_{\mathbb{R}^d}\varrho (x) dx=1$. For $l>0$, we define the rescaled mollifier $\varrho_\varepsilon(x)=\frac{1}{\varepsilon^d}\varrho(\frac{x}{\varepsilon})$ and for  any function $f\in L^1_{loc}(\Omega)$, its mollified version is defined as
$$f^\varepsilon(x)=(f*\varrho_{\varepsilon})(x)=\int_{B_\varepsilon(0)}f(x-y)\varrho_\varepsilon(y)dy,\ \ x\in \Omega^\varepsilon.$$
Eventually, we introduce some function spaces.\\
 $\mathbf{Sobelev \ spaces.}$  For $p\in [1,\,\infty]$, the notation $L^{p}(0,\,T;X)$ stands for the set of measurable functions on the interval $(0,\,T)$ with values in $X$ and $\|f(t,\cdot)\|_{X}$ belonging to $L^{p}(0,\,T)$. The classical Sobolev space $W^{k,p}(\Omega)$ is equipped with the norm $\|f\|_{W^{k,p}(\Omega)}=\sum\limits_{|\alpha| =0}^{k}\|D^{\alpha}f\|_{L^{p}(\Omega)}$.\\
 {\bf Besov spaces.} For any  $0<s<1$ and $1\leq p \leq \infty$, we define the Besov semi-norm $\norm{f}_{\dot{B}^s_{q,\infty}(\Omega)}$ and Besov norm $\norm{f}_{B^s_{q,\infty}(\Omega)}$ of $f\in \mathcal{S}^{'}$ as
	\begin{equation}\label{2.3}\|f\|_{\dot{B}_{q, \infty}^s(\Omega)}=\sup_{|y|>0}\f{\| f(x-y)-f(x)\|_{L^q(\Omega\cap (\Omega +\{y\}))}}{|y|^\alpha},
	\norm{f}_{B^s_{q,\infty}(\Omega)}=\norm{f}_{{L^q}(\Omega)}+\norm{f}_{\dot{B}^s_{q,\infty}(\Omega)},\end{equation}
	where $\Omega +\{y\}:=\{x\in \mathbb{R}^d| x=z+y, z\in \Omega\}$ with any $y\neq 0$ denoting the translation of the domain $\Omega $ according to $y$. Moreover, motivated by \cite{[CCFS],[WWY]} and \cite{[FW2018],[YWL]}, we can define the Onsager's critical space $B^\alpha _{p,c(\mathbb{N})}$ and Besov VMO type space $\underline{B}^\alpha_{p,VMO}$ as follows.
	
	We denote ${B}^\alpha _{q,c(\mathbb{N})}$ with $\alpha\in (0,1)$ and $1\leq q \leq \infty$ as the class of all tempered distributions $f$ satisfying
	\begin{equation}\label{2.4}
		\norm{f}_{{B}^\alpha _{q,\infty}(\Omega)}<\infty~ \text{and}~ 	\lim_{|y|\rightarrow 0} \f{\| f(x-y)-f(x)\|_{L^q(\Omega\cap (\Omega +\{y\}))}}{|y|^\alpha}=0.
	\end{equation}
And for any $\Omega^{'}\subset\subset\Omega$, we can formulate
the Besov-VMO type space  $\underline{B}^{\alpha}_{q,VMO}(\Omega)$  of function $f$ if it satisfies
\begin{equation}\label{2.5}
	\|f\|_{L^q(\Omega)}<\infty,\ \text{and}\ \lim_{\varepsilon\rightarrow0}\inf\f{1}{\varepsilon^{\alpha}}\B(\int_{\Omega^{'}} \fbxoo|f(x)-f(x-y)|^{q}dydx \B)^{\f{1}{q}}=0.
\end{equation}
Note that, combining the definitions of spaces $B^\alpha_{q,c(\mathbb{N})}$ and $\underline{B}^\alpha_{q,VMO}$ above, we can deduce that
\begin{equation}\label{2.6}
{B}^\alpha_{q,c(\mathbb{N})}(\Omega)\hookrightarrow \underline{B}^\alpha_{q,VMO}(\Omega)\hookrightarrow B^\alpha_{q,\infty}(\Omega),
\end{equation} for any $1\leq q\leq +\infty$ and $0<\alpha<1$.

Next, we collect some auxiliary lemmas which will be used in the present paper.
\begin{lemma}\label{lem2.1}(\cite{[BGSTW]})
Let $\alpha, \beta\in (0,1)$,  $ q\in [1,\infty]$,  and $k\in \mathbb{N}^+$.  Assume that  $f\in {B}^\alpha_{q,c(\mathbb{N})}(\Omega)$, $g\in \underline{B}^\beta_{q,VMO}(\Omega)$,  for any $\Omega^{'}\subset\subset\Omega$, then  we choose $\varepsilon$
sufficiently small such that $\varepsilon$-neighbourhood of $\Omega^{'}$ being in $\Omega$, as $\varepsilon\to 0$, there holds that
 \begin{enumerate}[(1)]
 \item $ \|f^{\varepsilon} -f \|_{L^{q}(\Omega^{'})}\leq C\varepsilon^{\alpha}\|f\|_{{B}^\alpha_{q,c(\mathbb{N})}(\Omega)}\leq \text{o}(\varepsilon^{\alpha})$;
   \item   $ \|\nabla^{k}f^{\varepsilon}  \|_{L^{q}(\Omega^{'})}\leq C\varepsilon^{\alpha-k}\|f\|_{{B}^\alpha_{q,c(\mathbb{N})}(\Omega)}\leq  \text{o}(\varepsilon^{\alpha-k})$;
       \item $ \|g^{\varepsilon} -g \|_{L^{q}(\Omega^{'})}\leq C\varepsilon^{\beta}\|g\|_{\underline{B}^\beta_{q,VMO}(\Omega)}\leq \text{o}(\varepsilon^{\beta})$;
   \item   $ \|\nabla^{k}g^{\varepsilon}  \|_{L^{q}(\Omega^{'})}\leq C\varepsilon^{\beta-k}\|g\|_{\underline{B}^\beta_{q,VMO}(\Omega)}\leq \text{o}(\varepsilon^{\beta-k})$.
 \end{enumerate}
\end{lemma}

\begin{lemma}\label{lem2.2}
	Let $ p,q,p_1,q_1,p_2,q_2\in[1,+\infty)$ with $\frac{1}{p}=\frac{1}{p_1}+\frac{1}{p_2},\frac{1}{q}=\frac{1}{q_1}+\frac{1}{q_2} $. Assume $f\in L^{p_1}(0,T;L^{q_1}(\Omega)) $ and $g\in L^{p_2}(0,T;L^{q_2}(\Omega))$.   For any $\Omega^{'}\subset\subset\Omega$, then  we choose $\varepsilon$
	sufficiently small such that $\varepsilon$-neighbourhood of $\Omega^{'}$ being in $\Omega$, as $\varepsilon\to 0$, there holds that
	\begin{equation}\label{b6}
		\|(fg)^{\varepsilon}-f^{\varepsilon} g^{\varepsilon}\|_{L^p(0,T;L^q(\Omega^{'}))}\rightarrow 0,
	\end{equation}
	and
	\begin{equation}\label{b7}
		\|(f\times g)^\varepsilon-f^{\varepsilon}\times g^{\varepsilon}\|_{L^p(0,T;L^q(\Omega^{'}))}\rightarrow 0.
	\end{equation}
\end{lemma}
\begin{proof}
	First, it follows from the triangle inequality that
	\begin{equation*}
		\begin{aligned}
			&\|(fg)^{\varepsilon}-f^{\varepsilon} g^{\varepsilon}\|_{L^p(0,T;L^q(\Omega^{'}))}\\
			\leq & C\left(\|(fg)^{\varepsilon}- (fg)\|_{L^p(0,T;L^q(\Omega^{'}))}+\|fg-f^{\varepsilon} g\|_{L^p(0,T;L^q(\Omega^{'}))}+\|f^{\varepsilon} g-f^{\varepsilon} g^{\varepsilon}\|_{L^p(0,T;L^q(\Omega^{'}))}\right)\\
			\leq &C\Big(\|(fg)^{\varepsilon}- fg\|_{L^p(0,T;L^q(\Omega^{'}))}+\|f-f^{\varepsilon}\|_{L^{p_1}(0,T;L^{q_1}(\Omega^{'}))}\|g\|_{L^{p_2}(0,T;L^{q_2}(\Omega^{'}))}\\
			&\ \ \ \ \ +\|f^{\varepsilon}\|_{L^{p_1}(0,T;L^{q_1}(\Omega^{'}))}\|g-g^{\varepsilon}\|_{L^{p_2}(0,T;L^{q_2}(\Omega^{'}))}\Big),
		\end{aligned}
	\end{equation*}
	then, together with the standard properties of  mollifiers, we can obtain \eqref{b6}.
	
	Furthermore, to conclude \eqref{b7}, without loss of generality, we assume $f=<f_1,f_2,f_3>$ and $g=<g_1,g_2,g_3>$.
	Then by the direct computation, we have
	\begin{equation}
		\begin{aligned}\label{b8}
			&\|(f\times g)^\varepsilon-f^{\varepsilon} \times g^{\varepsilon}\|_{L^p(0,T;L^q(\Omega^{'}))}=\Big\|\left(\left|
			\begin{matrix}
				\overrightarrow{i}&\overrightarrow{j}&\overrightarrow{k}\\
				f_1&f_2&f_3\\
				g_1&g_2&g_3\\
			\end{matrix}
			\right|\right)^\varepsilon-\left|
			\begin{matrix}
				\overrightarrow{i}&\overrightarrow{j}&\overrightarrow{k}\\
				f_1^\varepsilon&f_2^\varepsilon&f_3^\varepsilon\\
				g_1^\varepsilon&g_2^\varepsilon&g_3^\varepsilon\\
			\end{matrix}
			\right|\Big\|_{L^p(0,T;L^q(\Omega^{'}))}\\
			=&\Big\|\left((f_2g_3-f_3g_2)\overrightarrow{i}-(f_1g_3-f_3g_1)\overrightarrow{j}+(f_1g_2-f_2g_1)\overrightarrow{k}\right)^\varepsilon\\
			&~~~~-\left((f_2^\varepsilon g_3^\varepsilon-f_3^\varepsilon g_2^\varepsilon)\overrightarrow{i}-(f_1^\varepsilon g_3^\varepsilon-f_3^\varepsilon g_1^\varepsilon)\overrightarrow{j}+(f_1^\varepsilon g_2^\varepsilon-f_2^\varepsilon g_1^\varepsilon)\overrightarrow{k}\right)\Big\|_{L^p(0,T;L^q(\Omega^{'}))}\\
			=&\Big\|\left[\Big((f_2g_3)^\varepsilon-f_2^\varepsilon g_3^\varepsilon\Big)-\Big((f_3g_2)^\varepsilon-f_3^\varepsilon g_2^\varepsilon\Big)\right]\overrightarrow{i}-\left[\Big((f_1g_3)^\varepsilon-f_1^\varepsilon g_3^\varepsilon\Big)-\Big((f_3g_1)^\varepsilon-f_3^\varepsilon g_1^\varepsilon\Big)\right]\overrightarrow{j}\\
			&+\left[\Big((f_1g_2)^\varepsilon-f_1^\varepsilon g_2^\varepsilon\Big)-\Big((f_2g_1)^\varepsilon-f_2^\varepsilon g_1^\varepsilon\Big)\right]\overrightarrow{k}\Big\|_{L^p(0,T;L^q(\Omega^{'}))},
		\end{aligned}
	\end{equation}
	which together with triangle inequality and \eqref{b6} leads to \eqref{b7}.
\end{proof}
Next, we will state the Constantin-E-Titi type Commutator estimates as follows (see \cite{[YWL]} for torus).
\begin{lemma}	\label{lem2.3}
	Assume that $0<\alpha,\beta<1$, $1\leq p,q,p_{1},p_{2}\leq\infty$ and $\frac{1}{p}=\frac{1}{p_1}+\frac{1}{p_2}$, $\frac{1}{q}=\frac{1}{q_1}+\frac{1}{q_2}$.
	For any $\Omega^{'}\subset\subset\Omega$, then  we choose $\varepsilon$
	sufficiently small such that $\varepsilon$-neighbourhood of $\Omega^{'}$ being in $\Omega$, as $\varepsilon\to 0$, there holds that
	\begin{align} \label{cet}
		\|(fg)^{\varepsilon}- f^{\varepsilon}g^{\varepsilon}\|_{L^p(0,T;L^q(\Omega^{'}))} \leq \text{o}(\varepsilon^{\alpha+\beta}),	
	\end{align}
	provided one of the following three conditions holds
	\begin{enumerate}[(1)]
		\item  $f\in L^{p_1}(0,T;\underline{B}^{\alpha}_{q_{1},VMO} (\Omega))$, $g\in L^{p_2}(0,T;\underline{B}^{\beta}_{q_{2},VMO} (\Omega))$;
			\item  $f\in L^{p_1}(0,T;\underline{B}^{\alpha}_{q_{1},VMO} (\Omega))$, $g\in L^{p_2}(0,T;{B}^{\beta}_{q_{2},c(\mathbb{N})} (\Omega))$;
				\item  $f\in L^{p_1}(0,T;\underline{B}^{\alpha}_{q_{1},VMO} (\Omega))$, $g\in L^{p_2}(0,T;\dot{B}^{\beta}_{q_{2},\infty} (\Omega))$;
	
\end{enumerate}\end{lemma}
\begin{remark}
	The estimate as $\int_0^T\int_{\Omega} | (fg)^{\varepsilon}- f^{\varepsilon}g^{\varepsilon}||\nabla h |dxdt$ frequently appears in the study of energy (helicity) conservation  of fluid  equations
	(see \cite{[CET],[WYY],[YWW],[CY],[De Rosa],[WY]}).
	Therefore, it seems that this lemma is very helpful in this research for the weak solutions in critical Besov spaces.
\end{remark}
\begin{remark}\label{rem2.3} We would like to point out that this result also holds true if the norm $\underline{B}^{\alpha}_{q_{1},VMO}$  for $f$ is replaced by norm ${B}^{\alpha}_{q_{1},c(\mathbb{N})}$. \end{remark}
\begin{remark}
	Taking into account \eqref{b8}, it is easy to check that the results in this lemma are also hold for $\|(f\times g)^\varepsilon-f^{\varepsilon}\times g^{\varepsilon}\|_{L^p(0,T;L^q(\Omega^{'}))}\leq C\text{o}(\varepsilon^{\alpha+\beta})$.
\end{remark}
\begin{proof}
	First, we recall the following   identity observed  by Constantin-E-Titi   in \cite{[CET]} that
	$$\ba&(fg)^{\varepsilon}(x)- f^{\varepsilon}g^{\varepsilon}(x)\\
	=&
	\int_{B_\varepsilon(0)}
	[f(x-y)-f(x)][g(x-y)-g(x)]\varrho_\varepsilon(y)dy-
	(f-f^{\varepsilon})(g-g^{\varepsilon})(x),
	\ea$$
	which together with Minkowski inequality shows that
	\begin{equation}\label{2.16}
		\ba
		&\|(fg)^{\varepsilon}-f^{\varepsilon}g^{\varepsilon}\|_{L^q(\Omega^{'})}\\
		\leq& \left\|\int_{B_\varepsilon(0)}
		[f(x-y)-f(x)][g(x-y)-g(x)]\varrho_\varepsilon(y)dy\right\|_{L^q(\Omega^{'})}+\|(f-f^{\varepsilon})(g-g^{\varepsilon})\|_{L^q(\Omega^{'})}\\
		= & I_1+I_2.
		\ea
	\end{equation}
	(1)First,	using the H\"older inequality, we can estimate $I_1$ as
	\begin{equation}\label{2.17}\begin{aligned}  I_1&=\left\|\int_{B_\varepsilon(0)}
		[f(x-y)-f(x)][g(x-y)-g(x)]\varrho_\varepsilon(y)dy\right\|_{L^q(\Omega^{'})}\\
		&\leq C\left\|\left(\fbxoo\ \ \ \ |f(x-y)-f(x)|^{q_1}dy\right)^{\f{1}{q_1}}\left(\fbxoo\ \ \ \ |g(x-y)-g(x)|^{q_2}dy\right)^{\f{1}{q_2}}\right\|_{L^q(\Omega^{'})}\\
		&\leq C\left(\int_{\Omega^{'}}\fbxoo\ \ \ \ |f(x-y)-f(x)|^{q_1}dydx\right)^{\f{1}{q_1}}\left(\int_{\Omega^{'}}\fbxoo\ \ \ \ |g(x-y)-g(x)|^{q_2}dydx\right)^{\f{1}{q_2}}\\
		&\leq C\varepsilon^\alpha\|f\|_{\underline{B}^\alpha_{q_1,VMO}(\Omega)}\varepsilon^\beta\|g\|_{\underline{B}^\beta_{q_2,VMO}(\Omega)},
			\end{aligned} \end{equation}
	where we require $\f{1}{q}=\f{1}{q_1}+\f{1}{q_2}$.
	
	For term $I_2$, by virtue of the H\"older inequality and Lemma \ref{lem2.1}, we can obtain
	\begin{equation}\label{2.18}
		\ba
		I_2&=\|(f-f^{\varepsilon})(g-g^{\varepsilon})\|_{L^q(\Omega^{'})}\\
		&\leq C\|f-f^{\varepsilon}\|_{L^{q_1}(\Omega^{'})}\|g-g^{\varepsilon}\|_{L^{q_2}(\Omega^{'})}\\
		&\leq C\varepsilon^\alpha\|f\|_{\underline{B}^\alpha_{q_1,VMO}(\Omega)}\varepsilon^\beta\|g\|_{\underline{B}^\beta_{q_2,VMO}(\Omega)}.
		\ea
	\end{equation}
Then substituing \eqref{2.17} and \eqref{2.18} into \eqref{2.16}, integrating the resultant inequality over time and using the H\"older inequality, we have
\begin{equation}\label{2.19}
	\ba
	&\|(fg)^{\varepsilon}-f^{\varepsilon}g^{\varepsilon}\|_{L^p(0,T;L^q(\Omega^{'}))}\\
	\leq& C\varepsilon^{\alpha+\beta}\left\|\|f\|_{\underline{B}^\alpha_{q_1,VMO}(\Omega)}\|g\|_{\underline{B}^\beta_{q_2,VMO}(\Omega)}\right\|_{L^p(0,T)}\\
	\leq& C\varepsilon^{\alpha+\beta}\|f\|_{L^{p_1}(0,T;\underline{B}^\alpha_{q_1,VMO}(\Omega))}\|g\|_{L^{p_2}(0,T;\underline{B}^\alpha_{q_2,VMO}(\Omega))}\\
	\leq &C\text{o}(\varepsilon^{\alpha+\beta}),
	\ \text{as}\ \varepsilon\to0,\ea
\end{equation}
	where we have used the definition of Besov-VMO type space \eqref{2.5} and $\f{1}{p}=\f{1}{p_1}+\f{1}{p_2}$.
	
	(2) In this case, it is sufficient to prove $B^\beta_{q_2,c(\mathbb{N})}(\Omega)\hookrightarrow \underline{B}^\beta_{q_{2},VMO}(\Omega)$ for any $0<\beta<1$ and $1\leq q_2\leq +\infty$. It should be noted that when $\beta=\f{1}{3}$ and $q_2=3$, it has been proved by Bardos-Gwiazda-\'Swierczewska Gwiazda-Titi-Wiedemann in \cite{[BGSTW]}, for the convenience of the readers and the integrity of paper, we give the details of proof for the general case. Suppose that $g\in B^\beta_{q_2,c(\mathbb{N})}(\Omega)$, then there exists a function $l(z)\rightarrow0$ as $z\rightarrow0$ such that
	\begin{equation}\label{2.20}
		\ba
		\f{1}{|y|^\beta}\left(\int_{\Omega}|g(x-y)-g(x)|^{q_2}dx\right)^{\f{1}{q_2}}\leq l(y),
		\ea
	\end{equation}
	which implies
	$$\f{1}{|y|^{q_2\beta}}\int_{\Omega}|g(x-y)-g(x)|^{q_2}dx\leq l(y)^{q_2}.$$
Then for any $|y|<\varepsilon$ with $0<\varepsilon<1$, we have
\begin{equation}\label{2.21}
	\ba
	\f{1}{\varepsilon^{q_2\beta}}\int_{\Omega}|g(x-y)-g(x)|^{q_2}dx\leq l(y)^{q_2}.
	\ea
\end{equation}	
	Integrating the above inequality with respect to $y$ on $B_\varepsilon(0)$, we have
	\begin{equation}
		\ba
		\f{1}{\varepsilon^{q_2\beta}}\fbxoo\ \ \  \int_{\Omega}|g(x-y)-g(x)|^{q_2}dxdy\leq \fbxoo\ \  l(y)^{q_2}dy,
		\ea
	\end{equation}
	which  gives
	$$\f{1}{\varepsilon^{\beta}}\left(\fbxoo\ \ \  \int_{\Omega}|g(x-y)-g(x)|^{q_2}dxdy\right)^{\f{1}{q_2}}\leq \left(\fbxoo\ \  l(y)^{q_2}dy\right)^{\f{1}{q_2}}.$$
	Let $\overline{l}(\varepsilon)=\left(\fbxoo\ \  \ \  l(y)^{q_2}dy\right)^{\f{1}{q_2}}$ and it is easy to check that $\overline{l}(\varepsilon)\rightarrow0$ as $\varepsilon\rightarrow0$. Then using Fubini's Theorem, we can get
	$$\lim_{l\rightarrow0}\f{1}{\varepsilon^{\beta}}\left( \int_{\Omega^{'}}\fbxoo\ \ \ |g(x-y)-g(x)|^{q_2}dydx\right)^{\f{1}{q_2}} =0,$$
	which together with $g\in L^{q_2}(\Omega)$ means $g\in \underline{B}^\beta_{q_2,VMO}(\Omega)$. Then we have concluded the desired result.
	
	(3) By the same manner of  the proof in part (1),	using the H\"older inequality and Minkowski inequality, we can estimate $I_1$ as
	\begin{equation}\label{2.22}\begin{aligned}  I_1&=\left\|\int_{B_\varepsilon(0)}
			[f(x-y)-f(x)][g(x-y)-g(x)]\varrho_\varepsilon(y)dy\right\|_{L^q(\Omega^{'})}\\
			&\leq C\left\|\left(\fbxoo\ \ \ \ |f(x-y)-f(x)|^{q_1}dy\right)^{\f{1}{q_1}}\left(\fbxoo\ \ \ \ |g(x-y)-g(x)|^{q_2}dy\right)^{\f{1}{q_2}}\right\|_{L^q(\Omega^{'})}\\
			&\leq C\left(\int_{\Omega^{'}}\fbxoo\ \ \ \ |f(x-y)-f(x)|^{q_1}dydx\right)^{\f{1}{q_1}}\left(\fbxoo\ \ \ \ \|g(x-y)-g(x)\|^{q_2}_{L^{q_2}(\Omega^{'})}dy\right)^{\f{1}{q_2}}\\
			&\leq C\varepsilon^\alpha\|f\|_{\underline{B}^\alpha_{q_1,VMO}(\Omega)}\varepsilon^\beta\|g\|_{\dot{B}^\beta_{q_2,\infty}(\Omega)},
	\end{aligned} \end{equation}
	where we require $\f{1}{q}=\f{1}{q_1}+\f{1}{q_2}$.
	
	For term $I_2$, taking into account the  H\"older inequality and Lemma \ref{lem2.1}, we can obtain
	\begin{equation}\label{2.24}
		\ba
		I_2&=\|(f-f^{\varepsilon})(g-g^{\varepsilon})\|_{L^q(\Omega^{'})}\\
		&\leq C\|f-f^{\varepsilon}\|_{L^{q_1}(\Omega^{'})}\|g-g^{\varepsilon}\|_{L^{q_2}(\Omega^{'})}\\
		&\leq C\varepsilon^\alpha\|f\|_{\underline{B}^\alpha_{q_1,VMO}(\Omega)}\varepsilon^\beta\|g\|_{\dot{B}^\beta_{q_2,\infty}(\Omega)}.
		\ea
	\end{equation}
Then following the same manner of \eqref{2.19} and using the definition of Besov-VMO space, as $\varepsilon\to 0$,  we have
$$\|(fg)^{\varepsilon}-f^{\varepsilon}g^{\varepsilon}\|_{L^p(0,t;L^q(\Omega^{'}))}\leq \text{o}(\varepsilon^{\alpha+\beta}).$$
	Then the proof of this lemma is completed.
\end{proof}
Next, we present Moser type estimate in term of function in Besov type spaces.
\begin{lemma}\label{lem2.5}
	Let $\alpha\in (0,1)$ and $p\in [1,\infty]$.
\begin{enumerate}
	\item[(1)] If $f\in B^\alpha_{\infty,\infty}(\Omega)$ and $g\in B^\alpha_{p,\infty}(\Omega)$, then it holds that

		\begin{equation}\label{2.25}
		\|fg\|_{{{B}^\alpha_{p,\infty}}(\Omega)}\leq C\left(\|f\|_{L^\infty(\Omega)}\|g\|_{{{B}^\alpha_{p,\infty}}(\Omega)}+\|f\|_{{{B}^\alpha_{\infty,\infty}}(\Omega)}\|g\|_{L^p(\Omega)}\right);
	\end{equation}
\item[(2)] If $f\in B^\alpha_{\infty,c(\mathbb{N})}(\Omega)$ and $g\in B^\alpha_{p,c(\mathbb{N})}(\Omega)$, then it holds that
	\begin{equation}\label{2.26}
		\|fg\|_{{{B}^\alpha_{p,c(\mathbb{N})}}(\Omega)}\leq C\left(\|f\|_{L^\infty(\Omega)}\|g\|_{{{B}^\alpha_{p,c(\mathbb{N})}}(\Omega)}+\|f\|_{{{B}^\alpha_{\infty,c(\mathbb{N})}}(\Omega)}\|g\|_{L^p(\Omega)}\right);
	\end{equation}
\item[(3)] If $f\in B^\alpha_{\infty,c(\mathbb{N})}(\Omega)$ and $g\in \underline{B}^\alpha_{p,VMO}(\Omega)$, then it holds that
	\begin{equation}\label{2.27}
	\|fg\|_{{\underline{B}^\alpha_{p,VMO}}(\Omega)}\leq C\left(\|f\|_{L^\infty(\Omega)}\|g\|_{{\underline{B}^\alpha_{p,VMO}}(\Omega)}+\|f\|_{{{B}^\alpha_{\infty,c(\mathbb{N})}}(\Omega)}\|g\|_{L^p(\Omega)}\right).
\end{equation}
\end{enumerate}
\end{lemma}
\begin{proof}(1) By the triangle inequality,
we have
	\begin{equation}\label{2.28}
		\begin{aligned}
	&	|f(x-y)g(x-y)-f(x)g(x)|\\
		\leq & {|f(x-y)\B(g(x-y)-g(x)\B)|+|\B(f(x-y)-f(x)\B)g(x)|},
		\end{aligned}
	\end{equation}
	which together with the Minkowski inequality implies
	\begin{equation}\label{2.29}
		\begin{aligned}
	&\frac{\|f(x-y)g(x-y)-f(x)g(x)\|_{L^p(\Omega)}}	{|y|^\alpha}\\
\leq	& C\left(\frac{\|f(x-y)\B(g(x-y)-g(x)\B)\|_{L^p(\Omega)}}{|y|^\alpha}+\frac{\|\B(f(x-y)-f(x)\B)g(x)\|_{L^p(\Omega)}}{|y|^\alpha}\right)\\
\leq &C\left(\|f\|_{L^\infty(\Omega)}\|\frac{\|g(x-y)-g(x)\|_{L^p(\Omega)}}{|y|^\alpha}+\frac{\|f(x-y)-f(x)\|_{L^\infty(\Omega)}}{|y|^\alpha}\|g\|_{L^p(\Omega)}\right)\\
\leq &C \left(\|f\|_{L^\infty(\Omega)}\|g\|_{{{B}^\alpha_{p,\infty}}(\Omega)}+\|f\|_{{{B}^\alpha_{\infty,\infty}}(\Omega)}\|g\|_{L^p(\Omega)}\right).
		\end{aligned}
	\end{equation}
Moreover, since
\begin{equation}\label{2.30}\|fg\|_{L^p(\Omega)}\leq C\|f\|_{L^\infty(\Omega)}\|g\|_{L^p(\Omega)}.\end{equation}
Hence, combining \eqref{2.29} and \eqref{2.30}, we can obtain the desired estimate \eqref{2.25}.

(2) Based on \eqref{2.29}, taking the limits  as $|y|\rightarrow 0$ and using the definition of Onsager's critical space \eqref{2.4}, we have
\begin{equation}\label{2.31}
	\begin{aligned}
		&\lim\limits_{|y|\rightarrow0}\frac{\|f(x-y)g(x-y)-f(x)g(x)\|_{L^p(\Omega)}}	{|y|^\alpha}\\
		\leq C&\left(\|f\|_{L^\infty(\Omega)}	\lim\limits_{|y|\rightarrow0}\|\frac{\|g(x-y)-g(x)\|_{L^p(\Omega)}}{|y|^\alpha}+	\lim\limits_{|y|\rightarrow0}\frac{\|f(x-y)-f(x)\|_{L^\infty(\Omega)}}{|y|^\alpha}\|g\|_{L^p(\Omega)}\right)\\
	\leq C	& \left(\|f\|_{L^\infty(\Omega)}\|g\|_{{{B}^\alpha_{p,c(\mathbb{N})}}(\Omega)}+\|f\|_{{{B}^\alpha_{\infty,c(\mathbb{N})}}(\Omega)}\|g\|_{L^p(\Omega)}\right)\to 0,
	\end{aligned}
\end{equation}
which together with \eqref{2.30} and the definition of Onsager's crititcal space $B^\alpha _{p,c(\mathbb{N})}$ \eqref{2.4} again gives \eqref{2.26}.

(3) Recall \eqref{2.28}, for any $\Omega^{'}\subset\subset\Omega$, in light of the Minkowski inequality, we have
\begin{equation}\label{2.32}
	\ba
	&\int_{\Omega^{'}}\fbxoo\ \ \ \ |f(x-y)g(x-y)-f(x)g(x)|^pdydx\\
	\leq C&\left(\int_{\Omega^{'}}\fbxoo\ \ \ \ |f(x-y)|^p|g(x-y)-g(x)|^pdydx+\int_{\Omega^{'}}\fbxoo\ \ \ \ |f(x-y)-f(x)|^p|g(x)|^pdydx\right)\\
	\leq C&\left(\|f\|^p_{L^\infty(\Omega)}\int_{\Omega^{'}}\fbxoo\ \ \ \ |g(x-y)-g(x)|^pdydx+\int_{\Omega^{'}}|g(x)|^pdx\fbxoo\ \ \ \ \|f(x-y)-f(x)\|_{L^\infty(\Omega)}^pdy\right)\\
	\leq C& \left(\|f\|^p_{L^\infty(\Omega)}\varepsilon^{p\alpha}\|g\|_{\underline{B}^\alpha_{p,VMO}(\Omega)}^p+\|g\|_{L^p(\Omega)}^p\fbxoo\ \ \ \ \varepsilon^{\alpha p}\left(\f{\|f(x-y)-f(x)\|_{L^\infty(\Omega)}}{|y|^\alpha}\right)^pdy\right)\\
	\leq C&\left(\|f\|^p_{L^\infty(\Omega)}\varepsilon^{p\alpha}\|g\|_{\underline{B}^\alpha_{p,VMO}(\Omega)}^p+\|g\|_{L^p(\Omega)}^p \varepsilon^{\alpha p}\|f\|_{{B}^\alpha_{\infty,c(\mathbb{N})}(\Omega)}^p\right).
	\ea
\end{equation}
Then letting $\varepsilon\rightarrow 0$ and using the difinition of Besov-VMO space \eqref{2.5}, we conclude the desired result \eqref{2.27}.
\end{proof}
\begin{lemma}(\cite{[NNT2020]})\label{lem2.6}
 Assume that there exists a positive constant $M>0$ such that
	\begin{equation}\label{2.33}
		\|\rho\|_{L^\infty(0,T;B^\alpha_{\infty,\infty}(\Omega))}\leq M.
	\end{equation}
	For any $\alpha \in (0,1)$ and $0<\varepsilon^\alpha <\frac{1}{6mM}$,  there holds

(i) If $\rho^{\varepsilon}(t,x)\geq \frac{1}{2m}$ then $\rho  >\frac{1}{3m}$. In addition, for any $y\in B_\varepsilon(0)$, $\rho (x-y,t)>\frac{1}{6m}$.

(ii) If $\rho^{\varepsilon}(t,x)\leq  \frac{1}{m}$ then $\rho  <\frac{7}{6m}$. In addition, for any $y\in B_\varepsilon(0)$, $\rho (x-y,t)<\frac{4}{3m}$.
\end{lemma}
\begin{proof}
The proof can be found in \cite{[NNT2020]}, for the convenience of readers and self-contained of our paper, we give the details.	From \eqref{2.33}, we have
	\begin{equation}\label{2.34}
		|\rho(t,x-y)-\rho(t,x)|\leq M|y|^\alpha, ~\text{for~a.e.}~x\in \Omega^{\varepsilon},y\in B_\varepsilon(0), t\in (0,T).
	\end{equation}
Since
\begin{equation}\label{2.35}
	\rho^{\varepsilon}(t,x)=\int _{B_\varepsilon(0)}\B(\rho(t,x-y)-\rho(t,x)\B)\varrho_\varepsilon(y)dy +\rho(t,x).
\end{equation}
If $\rho^{\varepsilon}\geq \frac{1}{2m}, $ from \eqref{2.35}, we can obtain
\begin{equation}\label{2.36}
	\rho(t,x)= \rho(t,x)^\varepsilon-\int _{B_\varepsilon(0)}\B(\rho(t,x-y)-\rho(t,x)\B)\varrho_\varepsilon(y)dy \geq \frac{1}{2m}-M\varepsilon^\alpha>\frac{1}{3m}.
\end{equation}
Then together with \eqref{2.34}, it yields
\begin{equation}
	\rho(t,x-y)>\rho(t,x)-M\varepsilon^\alpha>\frac{1}{6m}.
\end{equation}
Similarly, if $\rho^{\varepsilon}(t,x)\leq \frac{1}{m}$, by \eqref{2.35}, we have
\begin{equation}
	\rho(t,x)= \rho^{\varepsilon}(t,x)+\int _{B_\varepsilon(0)}\B(\rho(t,x-y)-\rho(t,x)\B)\varrho_\varepsilon(y)dy\leq \frac{1}{m}+M\varepsilon^\alpha<\frac{7}{6m},
\end{equation}
which together with \eqref{2.34} implies
\begin{equation}
	\rho(t,x-y)\leq \rho(t,x)+M\varepsilon^\alpha\leq \frac{4}{3m}.
\end{equation}
Then we have completed the proof of this lemma.
\end{proof}
Hence, Before we state our result, for the convenience of readers, we first give the definition of the weak solutions to compressible Euler equations \eqref{CEuler}.

\begin{definition}\label{ceulerdefi}
	A pair ($\rho,v$) is called a weak solution to  the   compressible Euler equations  \eqref{CEuler}  if ($\rho,v$) satisfies
	
	\begin{enumerate}[(i)]
		\item For any test function $\Phi\in C_0^\infty((0,t)\times \Omega)$, there holds
		\begin{equation}\label{2.43}\int^T_0\int_{\Omega}\rho(t,x)\partial_t \Phi(t,x)+ \rho(t,x)v(t,x)\nabla \Phi(t,x)dxdt=0\end{equation}
		\item
		For any  test vector field $\Psi\in C_{0}^{\infty}((0,t)\times\Omega)$, there holds
		\begin{equation}\label{2.44}
			\begin{aligned}
				\int_{0}^{T}\int_{\Omega}\rho(t,x)v(t,x)&\partial_{t}\Psi(t,x)+\B(\rho(t,x)v(t,x)\otimes v(t,x)\B) \nabla\Psi(t,x)\\
				+ &
				p(\rho )(t,x)\text{div} \Psi(t,x)dxdt=0.
		\end{aligned}	\end{equation}
		\item
		The energy inequality holds
		\begin{equation}\label{energyineq}
			\begin{aligned}
				\mathcal{E}(t)   \leq 	\mathcal{E}(0), 		\end{aligned}\end{equation}
		where $	\mathcal{E}(t)=\int_{\Omega}\B( \frac{1}{2}\rho |v|^2+G(\rho) \B) dx$, with $G(\rho)=\rho\int_{1}^\rho \f{p(s)}{s^2}ds$.
		
	\end{enumerate}
\end{definition}

\section{ Helicity  conservation for the compressible Euler equations}
This section is devoted to the proof of  the helicity conservation  of weak solutions for  compressible Euler equations with density containing vacuum in a bounded domain. We will divide the proof into two steps. First, we need the following proposition.
\begin{prop}\label{propo1}
	Let $(\rho,v)$ be a weak solution of compressible Euler equations \eqref{CEuler} in the sense of Definition \ref{ceulerdefi}. Assume that
	\begin{enumerate}[(i)]
		\item (Interior Besov regularity) if $(\rho,v)$ satisfy
		\begin{equation}\label{3.1}
			\ba
		&	0\leq \rho \leq c<\infty,~\rho \in L^3(0,T;B^{\f{1}{3}}_{3,c(\mathbb{N})}(\Omega))\cap  L^\infty(0,T;B^{\frac{1}{3}}_{\infty,\infty}(\Omega)), \\
		& v \in L^3(0,T;\underline{B}^{\frac{1}{3}}_{3,VMO}(\Omega))
		~ \text{and}~  \omega\in L^{\f{3}{2}}(0,T;L^{\f{3}{2}}(\Omega)) ,\ea\end{equation}
	\item (Near the vacuum)for any $\delta>0$, if
	\begin{equation}\label{3.3}
		\begin{aligned}
		&\rho^{-1}{\bf 1}_{\rho \leq \delta}\in L^{r_1}(0,T;L^{r_1}(\Omega)), {\bf 1}_{\rho \leq \delta}\partial_t \rho\in L^{r_2}(0,T;L^{r_2}(\Omega)),\text{and}~ {\bf 1}_{\rho \leq \delta}\nabla \rho\in L^{r_3}(0,T;L^{r_3}(\Omega)),
\end{aligned}	\end{equation}
	\end{enumerate}
	 	for some positive constant $c,r_1,r_2,r_3$ and some sufficiently small constant $\delta>0$ such that $\frac{1}{r_1}+\frac{1}{r_2}=\frac{2}{3}$ and $\frac{1}{r_1}+\frac{1}{r_3}=\frac{1}{3}$. Assume further that $p\in C^2[c_1,c_2]$, then $(\rho,v)$ is also a weak solution of the following system  in the sense of distributions on $(0,T)\times \Omega$
	 	\be\left\{\ba\label{rCEuler2}
	 	&\rho_t+\nabla \cdot (\rho v)=0, \\
	 	& v_{t} + v\cdot\nabla v+\nabla
	 	P(\rho )=0,
	 	\ea\right.\ee
	 	where  $P(\rho )= \int_{1}^{\rho}\f{p'(s)}{s}ds$.
	 	
\end{prop}
\begin{proof}
	First, to control the possible vacuum, we introduce the following cut-off function $\phi_m(z) $ such that
	\begin{equation}\label{3.4}
		0\leq \phi_m(z)\in C^\infty([0,\infty)),~\phi_m(z)=1~\text{for\ any}~z> \frac{1}{m},~\phi_m(z)=0~\text{for\ any}~z < \frac{1}{2m}.		
	\end{equation}
It is easy to check that $\phi_m(z)\rightarrow1$  as $m\rightarrow+\infty$ pointwise in $\Omega $ and for any $z\in \Omega$, we have $z\nabla \phi_m(z)\leq C$. Then let $\varphi(t,x)\in C^\infty_0((0,T)\times \Omega)$ with the support of which is strictly contained in an open set $(t_1,t_2)\times \Omega^{'}\subset\subset(0,T)\times \Omega$, due to $(i)$ in definition \ref{ceulerdefi}, $\left(\varphi(t,x)\f{\phi_m(\rho^\varepsilon)}{\rho^\varepsilon}\right)^\varepsilon$ is not only compactly supported in both space and time, but also smooth in space and weakly Lipschitz continuous in time, which allows us to take it as a test function in the proof of this propositon.

Let $\Psi=\left(\varphi(t,x)\f{\phi_m(\rho^\varepsilon)}{\rho^\varepsilon}\right)^\varepsilon$ in \eqref{2.44} as the test function, one has
	\begin{eqnarray}\label{3.5}
	\int_0^T\int _{\Omega^{'}}\B(	(\rho v)^{\varepsilon}_t+\text{div}(\rho v\otimes v)^\varepsilon +\nabla p(\rho)^\varepsilon\B) \varphi(t,x)\f{\phi_m(\rho^\varepsilon)}{\rho^\varepsilon} dxdt=0.
	\end{eqnarray}
Then using appropriate commutators, the above equality can be rewritten as
	\begin{equation}\label{3.6}
		\begin{aligned}
			&\int_0^T\int_{\Omega^{'}}\B((\rho^{\varepsilon} v^{\varepsilon})_t +\text{div}((\rho v)^{\varepsilon}\otimes v^{\varepsilon})+\nabla p(\rho^{\varepsilon})\B)\varphi(t,x)\f{\phi_m(\rho^{\varepsilon})}{\rho^{\varepsilon}} dxdt\\
			=&-\int_0^T\int_{\Omega^{'}}\Big([(\rho v)^{\varepsilon}-(\rho^{\varepsilon} v^{\varepsilon})]_t+\text{div}[(\rho v\otimes v)^{\varepsilon}-(\rho v)^{\varepsilon} \otimes v^{\varepsilon}]\\
			&\ \ \ \ \ \ \ \ \ \ \ \ \ \ \ \ \ \ \ \ \ \ \ \ \ \ \ \ \ \ +\nabla [p(\rho)^{\varepsilon}-p(\rho^{\varepsilon})]\Big)\varphi(t,x)\f{\phi_m(\rho^{\varepsilon})}{\rho^{\varepsilon}} dxdt.
		\end{aligned}
	\end{equation}
	By direct computation and taking into account the  mollified version of the mass
	equation
	\begin{equation}\label{3.7}
		\rho^{\varepsilon}_t+\text{div}(\rho v)^{\varepsilon}=0,
	\end{equation}
	one has
	\begin{equation}\label{3.8}
		\begin{aligned}
		&\int_0^T\int_{\Omega^{'}}\B(	\rho^{\varepsilon} v^{\varepsilon}_t+\rho^{\varepsilon} v^{\varepsilon} \cdot \nabla v^{\varepsilon}+\nabla p(\rho^{\varepsilon})\B)\varphi(t,x)\f{\phi_m(\rho^{\varepsilon})}{\rho^{\varepsilon}} dxdt\\
			=&-\int_0^T\int_{\Omega^{'}}\Big([ (\rho v)^{\varepsilon}-\rho^{\varepsilon} v^{\varepsilon}] \cdot \nabla v^{\varepsilon}
			+[(\rho v)^{\varepsilon}-(\rho^{\varepsilon} v^{\varepsilon})]_t\\
			&+\text{div}[(\rho v\otimes v)^{\varepsilon}-(\rho v)^{\varepsilon} \otimes v^{\varepsilon}]+\nabla [p(\rho)^{\varepsilon}-p(\rho^{\varepsilon})]\Big)\varphi(t,x)\f{\phi_m(\rho^{\varepsilon})}{\rho^{\varepsilon}} dxdt.
		\end{aligned}
	\end{equation}
	Setting $P(\rho^{\varepsilon} )= \int_{1}^{\rho^{\varepsilon}}\f{p'(s)}{s}ds$, we arrive at
	\begin{equation}\label{3.9}
		\begin{aligned}
			&\int_0^T\int_{\Omega^{'}}\left(	v^{\varepsilon}_t+\omega^{\varepsilon} \times v^{\varepsilon}+\nabla\left( P(\rho^{\varepsilon})+\frac{1}{2}|v^{\varepsilon}|^2\right)\right)\varphi(t,x){\phi_m(\rho^{\varepsilon})} dxdt\\
			=&-\int_0^T\int_{\Omega^{'}}\Big(	[ (\rho v)^{\varepsilon}-\rho^{\varepsilon} v^{\varepsilon}] \cdot \nabla v^{\varepsilon}
		+[(\rho v)^{\varepsilon}-(\rho^{\varepsilon} v^{\varepsilon})]_t\\
			&+\text{div}[(\rho v\otimes v)^{\varepsilon}-(\rho v)^{\varepsilon} \otimes v^{\varepsilon}]
			+\nabla [p(\rho)^{\varepsilon}-p(\rho^{\varepsilon})]\Big)\varphi(t,x)\f{\phi_m(\rho^{\varepsilon})}{\rho^{\varepsilon}} dxdt\\
			=&R_1+R_2+R_3+R_4,
		\end{aligned}
	\end{equation}
	where we have used the indentity $v^{\varepsilon}\cdot \nabla v^{\varepsilon}=\frac{1}{2}\nabla |v^{\varepsilon}|^2 +\omega^{\varepsilon} \times v^{\varepsilon}$ with $\omega =\text{curl} v$.
	
We denote the terms on the LHS of \eqref{3.9} by $L_1-L_3$. Then, for the first term $L_1$, using the integration by parts, we have
\begin{equation}\label{3.10}
	\begin{aligned}
		L_1&=	\int_0^T\int_{\Omega^{'}} v^{\varepsilon}_t \varphi(t,x) \phi_m(\rho^{\varepsilon})dxdt\\
		&=-\int_0^T\int_{\Omega^{'}} v^{\varepsilon}\varphi(t,x)_t {\phi_m(\rho^{\varepsilon})}dxdt -\int_0^T\int_{\Omega^{'}} v^{\varepsilon} \varphi(t,x)\phi_m^{'}(\rho^{\varepsilon})\partial_t \rho^{\varepsilon}dxdt\\
		&=L_{11}+L_{12}.
	\end{aligned}
\end{equation}
To estimate $L_{12}$, by H\"older's inequality and the fact that $\rho^{\varepsilon} \phi^{'}_m(\rho^{\varepsilon})\leq C$, we have
\begin{equation}\label{3.11}
	\begin{aligned}
		|L_{12}|\leq& \int_0^T\int_{\Omega^{'}\cap \{\f{1}{2m}\leq \rho^{\varepsilon}\leq \f{1}{m}\}} |v^{\varepsilon}\varphi(t,x)   \frac{1}{\rho^{\varepsilon}} \partial_{t}\rho^{\varepsilon}| dxdt\\
		\leq& C\|v^{\varepsilon}\|_{L^3(0,T;L^3(\Omega^{'})}\|\frac{1}{\rho^{\varepsilon}}{\bf 1}_{\{\frac{1}{2m}\leq \rho^{\varepsilon}\leq \frac{1}{m}\}}\|_{L^{r_1}(0,T;L^{r_1}(\Omega^{'}))}\|\partial_t \rho^{\varepsilon}{\bf 1}_{\{\frac{1}{2m}\leq \rho^{\varepsilon}\leq \frac{1}{m}\}}\|_{L^{r_2}(0,T;L^{r_2}(\Omega^{'}))},
	\end{aligned}
\end{equation}
where $\f{1}{r_1}+\f{1}{r_2}=\f{2}{3}$.
Moreover, by virtue of the Minkowski inequality and Lemma \ref{lem2.6}, we can deduce that
\begin{equation}\label{3.12}
	\begin{aligned}
		&\|\partial_t \rho^{\varepsilon}{\bf 1}_{\{\frac{1}{2m}\leq \rho^{\varepsilon}\leq \frac{1}{m}\}}\|_{L^{r_2}(0,T;L^{r_2}(\Omega^{'}))}\\
		\leq& \|\int_{B_\varepsilon(0)}\partial_t \rho(x-y,t){\bf 1}_{\{\frac{1}{2m}\leq \rho^{\varepsilon}\leq \frac{1}{m}\}}\varrho_\varepsilon(y) dy\|_{L^{r_2}(0,T;L^{r_2}(\Omega^{'}))}\\
	\leq	& C\int_{B_\varepsilon(0)}\|\partial_t \rho(x-y){\bf 1}_{\{\frac{1}{6m}\leq \rho\leq \frac{4}{3m}\}}\|_{L^{r_2}(0,T;L^{r_2})(\Omega^{'})}\varrho_\varepsilon(y) dy\\
		\leq& C\|\partial_t \rho(x-y){\bf 1}_{\{ \rho\leq \delta\}}\|_{L^{r_2}(0,T;L^{r_2}(\Omega))}.
\end{aligned}\end{equation}
Similarly, we also have
\begin{equation}\label{3.13}
	\begin{aligned}
		&\|\frac{1}{\rho^{\varepsilon}}{\bf 1}_{\{\frac{1}{2m}\leq \rho^{\varepsilon}\leq \frac{1}{m}\}}\|_{L^{r_1}(0,T;L^{r_1})(\Omega^{'})}\\
		\leq& C\left(\|(\f{1}{\rho}-\f{1}{\rho^\varepsilon}){\bf 1}_{\{\frac{1}{2m}\leq \rho^{\varepsilon}\leq \frac{1}{m}\}}\|_{L^{r_1}(0,T;L^{r_1}(\Omega^{'}))}+ \|\frac{1}{\rho}{\bf 1}_{\{\frac{1}{3m}\leq \rho\leq \frac{7}{6m}\}}\|_{L^{r_1}(0,T;L^{r_1}(\Omega))}\right)\\
		\leq& C\left(m^2\|(\rho^\varepsilon-\rho)\|_{L^{r_1}(0,T;L^{r_1}(\Omega^{'}))}+\|\frac{1}{\rho}{\bf 1}_{\rho\leq \delta}\|_{L^{r_1}(0,T;L^{r_1}(\Omega))}\right).
	\end{aligned}
\end{equation}
Thus plugging \eqref{3.12} and \eqref{3.13} into \eqref{3.11} and using the Lebesgue dominated convergence theorem, we have
\begin{equation}\label{3.14}
	\lim\limits_{m\rightarrow +\infty}\lim\limits_{\varepsilon\rightarrow 0}|L_{12}|=0~\text{and}~ \lim\limits_{m\rightarrow +\infty}\lim\limits_{\varepsilon\rightarrow 0}L_{11}=-\int_0^T\int_{\Omega} v\varphi(t,x)_t dxdt,
\end{equation}
which implies
\begin{equation}\label{3.15}
	\lim\limits_{m\rightarrow +\infty}\lim\limits_{\varepsilon\rightarrow 0}L_1=-\int_0^T\int_{\Omega} v\varphi(t,x)_t dxdt.
\end{equation}
For the second term $L_2$ on the LHS of \eqref{3.9}, using the Lebesgue domainated convergence theorem, we have
\begin{equation}\label{3.16}
	\lim\limits_{m\rightarrow +\infty}\lim\limits_{\varepsilon\rightarrow 0}	L_2=\int_0^T\int_\Omega \omega\times v \varphi(t,x) dxdt.
\end{equation}
By integration by parts, the third term $L_3$ on the LHS of \eqref{3.9} can be handled as
\begin{equation}\label{3.17}
	\begin{aligned}
		L_3&=-\int_0^T\int_{\Omega^{'}} \left( P(\rho^{\varepsilon})+\frac{1}{2}|v^{\varepsilon}|^2\right)\left(\phi_m(\rho^{\varepsilon})\text{div}\varphi(t,x)  +\varphi(t,x) \cdot\phi_m^{'}(\rho^{\varepsilon})\nabla \rho^{\varepsilon} \right)dxdt\\
		&=L_{31}+L_{32}.
	\end{aligned}
\end{equation}
Similar as \eqref{3.10}, with $\f{1}{r_1}+\f{1}{r_3}=\f{1}{3}$, we can obtain
\begin{equation}\label{3.18}
	\begin{aligned}
&|L_{32}|\\\le&q 	|-\int_0^T\int_{\Omega^{'}} \left( P(\rho^{\varepsilon})+\frac{1}{2}|v^{\varepsilon}|^2\right)\varphi(t,x) \cdot\phi_m^{'}(\rho^{\varepsilon})\nabla \rho^{\varepsilon} dxdt|\\
=&	|	\int_0^T\int_{\Omega^{'}} \left( P(\rho^{\varepsilon})+\frac{1}{2}|v^{\varepsilon}|^2\right)\rho^{\varepsilon}\phi_m^{'}(\rho^{\varepsilon})\frac{1}{\rho^{\varepsilon}}{\bf 1}_{\{\frac{1}{2m}\leq \rho^{\varepsilon}\leq \frac{1}{m}\}}\varphi(t,x) \cdot\nabla \rho^{\varepsilon}{\bf 1}_{\{\frac{1}{2m}\leq \rho^{\varepsilon}\leq \frac{1}{m}\}} dxdt|\\
\leq& C\| P(\rho^{\varepsilon})+\frac{1}{2}|v^{\varepsilon}|^2\|_{L^{\frac{3}{2}}(0,T;L^{\frac{3}{2}}(\Omega^{'}))}\|\frac{1}{\rho^{\varepsilon}}{\bf 1}_{\{\frac{1}{2m}\leq \rho^{\varepsilon}\leq \frac{1}{m}\}}\|_{L^{r_1}(0,T;L^{r_1}(\Omega^{'}))}\|\nabla \rho^{\varepsilon}{\bf 1}_{\{\frac{1}{2m}\leq \rho^{\varepsilon}\leq \frac{1}{m}\}}\|_{L^{r_3}(0,T;L^{r_3}(\Omega^{'}))}\\
		\leq& C\| P(\rho)+\frac{1}{2}|v|^2\|_{L^{\frac{3}{2}}(0,T;L^{\frac{3}{2}}(\Omega))}\|\frac{1}{\rho}{\bf 1}_{\{\rho\leq \delta\}}\|_{L^{r_1}(0,T;L^{r_1}(\Omega))}\|\nabla \rho{\bf 1}_{\{\rho\leq \delta\}}\|_{L^{r_3}(0,T;L^{r_3}(\Omega))},
	\end{aligned}
\end{equation}
which together with the  Lebesgue dominated convergence theorem gives
\begin{equation}\label{3.20}
	\ba
\lim\limits_{m\rightarrow +\infty}\lim\limits_{\varepsilon\rightarrow 0}	L_{31}=	-\int_0^T\int_{\Omega} \left( P(\rho)+\frac{1}{2}|v|^2\right)\text{div}\varphi(t,x) dxdt,\
	\lim\limits_{m\rightarrow +\infty}\lim\limits_{\varepsilon\rightarrow 0}	L_{32}=0.
	\ea
\end{equation}
In conclusion, we have
\begin{equation}\label{3.21}
	\lim\limits_{m\rightarrow +\infty}\lim\limits_{\varepsilon\rightarrow 0} LHS=\int_0^T\int_{\Omega} -v\varphi(t,x)_t+ (\omega\times v )\varphi(t,x) -\left( P(\rho)+\frac{1}{2}|v|^2\right) \text{div}\varphi(t,x) dxdt .
\end{equation}
Next, we are ready to estimate the terms $R_1-R_4$ on the RHS of \eqref{3.9}. For $R_1$, using the H\"older inequality, Lemma \ref{lem2.2} and  Lemma \ref{lem2.3}, we see that
\begin{equation}\label{3.22}
	\begin{aligned}
		|R_1|&=\left|-\int_0^T\int_{\Omega^{'}}	[ (\rho v)^{\varepsilon}-\rho^{\varepsilon} v^{\varepsilon}] \cdot \nabla v^{\varepsilon} \varphi(t,x)\f{\phi_m(\rho^{\varepsilon})}{\rho^{\varepsilon}} dxdt\right|\\
			&\leq mC\|\varphi\|_{L^3(0,T;L^3(\Omega^{'}))}\|(\rho v)^{\varepsilon}-\rho^{\varepsilon} v^{\varepsilon}\|_{L^3(0,T;L^3(\Omega^{'}))}\|\nabla v^{\varepsilon}\|_{L^3(0,T;L^3(\Omega^{'}))}\\
		&\leq mC\varepsilon^{\f{2}{3}}\|\rho\|_{L^\infty(0,T;B^{\frac{1}{3}}_{\infty,\infty}(\Omega))}\|v\|_{L^3(0,T;\underline{B}^{\frac{1}{3}}_{3,VMO}(\Omega))}\varepsilon^{-\f{2}{3}}\|v\|_{L^3(0,T;\underline{B}^{\frac{1}{3}}_{3,VMO}(\Omega))}\\
		&\leq mC\|\rho\|_{L^\infty(0,T;B^{\frac{1}{3}}_{\infty,\infty}(\Omega))}\|v\|_{L^3(0,T;\underline{B}^{\frac{1}{3}}_{3,VMO}(\Omega))}^2.
	\end{aligned}
\end{equation}
This, after taking the limits as $\varepsilon\rightarrow 0$ firstly and $m\rightarrow +\infty$ secondly and using the definition of  Besov-VMO type space $\underline{B}^\alpha_{p,VMO}$ \eqref{2.5}, we have
\begin{equation}\label{3.23}
	\lim\limits_{m\rightarrow +\infty}\lim\limits_{\varepsilon\rightarrow 0} R_1=0.
\end{equation}
For $R_2$, using the integration by parts,  the mass equation, Lemma \ref{lem2.2} and Lemma \ref{lem2.3}, one has
\begin{equation}\label{3.24}
	\begin{aligned}
		&|R_2|\\
		=&\left|	\int_0^T\int_{\Omega^{'}}[(\rho v)^{\varepsilon}-(\rho^{\varepsilon} v^{\varepsilon})] \left(\varphi(t,x)_t\frac{\phi_m(\rho^{\varepsilon})}{\rho^{\varepsilon}}+\varphi(t,x) \f{\phi_m^{'}(\rho^{\varepsilon})\partial_t \rho^{\varepsilon}}{\rho^{\varepsilon}}-\varphi(t,x) \phi_m(\rho^{\varepsilon})\frac{\partial_t \rho^{\varepsilon}}{(\rho^{\varepsilon})^2}\right)dxdt\right |\\
		=&\left|	\int_0^T\int_{\Omega^{'}}[(\rho v)^{\varepsilon}-(\rho^{\varepsilon} v^{\varepsilon})] \left(\varphi(t,x)_t\frac{\phi_m(\rho^{\varepsilon})}{\rho^{\varepsilon}}-\varphi(t,x) \rho^{\varepsilon}\phi_m^{'}(\rho^{\varepsilon})\frac{\text{div}(\rho v)^{\varepsilon}}{(\rho^{\varepsilon})^2}+\varphi(t,x) \phi_m(\rho^{\varepsilon})\frac{\text{div} (\rho v)^{\varepsilon}}{(\rho^{\varepsilon})^2}\right) dxdt \right|\\
		\leq& C\|(\rho v)^{\varepsilon}-\rho^{\varepsilon} v^{\varepsilon}\|_{L^3(0,T;L^3(\Omega^{'}))}\left(m\|\varphi_t\|_{L^{\frac{3}{2}}(0,T;L^{\frac{3}{2}}(\Omega^{'}))}+m^2\|\varphi\|_{L^3(0,T;L^3(\Omega^{'}))}\|\nabla (\rho v)^{\varepsilon}\|_{L^3(0,T;L^3(\Omega^{'}))}\right)\\
		\leq& C\varepsilon^{{\frac{2}{3}}}\|\rho\|_{L^\infty(0,T;B^{\frac{1}{3}}_{\infty,\infty}(\Omega))}\|v\|_{L^3(0,T;\underline{B}^{\frac{1}{3}}_{3,VMO}(\Omega))}\\
		&\left(m+m^2 \varepsilon^{-{\frac{2}{3}}}\left(\|\rho\|_{L^\infty(0,T;L^\infty(\Omega))}\|v\|_{L^3(0,T;\underline{B}^{\frac{1}{3}}_{3,VMO}(\Omega))}+\|\rho\|_{L^\infty(B^{\frac{1}{3}}_{\infty,\infty}(\Omega))}\|v\|_{L^3(0,T;L^3(\Omega))}\right)\right),
	\end{aligned}
\end{equation}
which follows from the definitions of  Besov-VMO type space \eqref{2.5} that
\begin{equation}\label{3.25}
	\lim\limits_{m\rightarrow +\infty}\lim\limits_{\varepsilon\rightarrow 0} R_2=0.
\end{equation}
Follow the same manner of derivation of $R_2$, one can infer that
\begin{equation}\label{3.26}
	\begin{aligned}
		&|R_3|\\
		=&\big|	\int_0^T\int_{\Omega^{'}}[(\rho v\otimes v)^{\varepsilon}-(\rho v)^{\varepsilon}\otimes v^{\varepsilon}]\B(\nabla \varphi(t,x)\frac{\phi_m(\rho^{\varepsilon})}{\rho^{\varepsilon}} +\varphi(t,x)\frac{\rho^{\varepsilon}\phi_m^{'}(\rho^{\varepsilon})\nabla \rho^{\varepsilon}-\phi_m(\rho^{\varepsilon})\nabla \rho^{\varepsilon}}{(\rho^{\varepsilon})^2}\B)dxdt\Big|\\
		\leq &C\|(\rho v\otimes v)^{\varepsilon}-(\rho v)^{\varepsilon} \otimes v^{\varepsilon}\|_{L^{\frac{3}{2}}(0,T;L^{\frac{3}{2}}(\Omega^{'}))}\Big(m\|\nabla \varphi\|_{L^3(0,T;L^3(\Omega^{'}))}+m^2\|\nabla \rho^{\varepsilon}\|_{L^\infty(0,T;L^\infty(\Omega^{'}))}\|\varphi\|_{L^3(0,T;L^3(\Omega^{'}))}\Big)\\
		\leq &C \varepsilon^{{\frac{2}{3}} }\left(\|\rho\|_{L^\infty(0,T;L^\infty(\Omega))}\|v\|_{L^3(0,T;\underline{B}^\alpha_{3,VMO}(\Omega))}+\|\rho\|_{L^\infty(0,T;B^{\frac{1}{3}}_{\infty,\infty}(\Omega))}\|v\|_{L^3(0,T;L^3(\Omega))}\right)\\
		&\ \ \  \|v\|_{L^3(0,T;\underline{B}^{\f{1}{3}}_{3,VMO}(\Omega))}\left(m+m^2\varepsilon^{-{\frac{2}{3}}}\|\rho\|_{L^\infty(0,T;B^{\frac{1}{3}}_{\infty,\infty}(\Omega))}\right)\\
	\end{aligned}
\end{equation}
After taking the limits as $\varepsilon\rightarrow 0$ and $m\rightarrow +\infty$ and using the definitions of  Besov-VMO type space \eqref{2.5}, we can obtain
\begin{equation}\label{3.27}
	\lim\limits_{m\rightarrow +\infty}\lim\limits_{\varepsilon\rightarrow 0} R_3=0.
\end{equation}
	For the last term $R_4$, due to the Taylor's expansion  that, for any $ s_0\geq \f{1}{3m}$, $s\leq \|\rho\|_{L^\infty(0,T;L^\infty)}$,  there holds
$$|p(s)-p(s_0)-p^{'}(s_0)(s-s_0)|\leq C(s-s_0)^2,$$
where the constant $C=\sup\{p^{''}(\xi)|:\f{1}{2m}\leq \xi\leq \|\rho\|_{L^\infty(0,T;L^\infty)}\}$ .
Therefore
\begin{equation}\label{3.28}
	|p(\rho^{\varepsilon}(t,x))-p(\rho(t,x))-p^{'}(\rho (t,x))(\rho^{\varepsilon}(t,x)-\rho(t,x))|\leq C|\rho^{\varepsilon}(t,x)-\rho(t,x)|^2,
\end{equation}
and similarly, we also have
\begin{equation}\label{3.29}
	|p(\rho(t,y))-p(\rho(t,x))-p^{'}(\rho (t,x))(\rho (t,y)-\rho (t,x))|\leq C|\rho (t,y)-\rho (t,x)|^2.
\end{equation}
Applying convolution with respect to $y$ to the last inequality and invoking Jensen's inequality, we can obtain
\begin{equation}\label{3.30}
	\begin{aligned}
		|p^\varepsilon(\rho(t,x))-p(\rho(t,x))-p^{'}(\rho (t,x))(\rho^{\varepsilon} 	 (t,x)-\rho (t,x))|\leq C(\rho (t,\cdot)-\rho (t,x))^2*\varrho_\varepsilon,
	\end{aligned}
\end{equation}
which together with \eqref{3.28} gives
\begin{equation}\label{pi}
	\begin{aligned}
		|p(\rho^{\varepsilon}(t,x))-p^\varepsilon(\rho(t,x))|&\leq C|\rho^{\varepsilon}(t,x)-\rho(t,x)|^2+C(\rho (t,\cdot)-\rho (t,x))^2*\varrho_\varepsilon.
	\end{aligned}
\end{equation}
Then after the integration by parts, $R_4$ can be handlded as
\begin{equation}\label{3.32}
	\begin{aligned}
		R_4&=\int_0^T\int_{\Omega^{'}} [p(\rho^{\varepsilon})-p^\varepsilon(\rho)]\left(\frac{\phi_m(\rho^{\varepsilon})}{\rho^{\varepsilon}}\text{div} \varphi(t,x)-\varphi(t,x)\cdot \frac{\phi_m(\rho^{\varepsilon})\nabla \rho^{\varepsilon}}{(\rho^{\varepsilon})^2}\right)dxdt\\
		&\leq C\|\rho^{\varepsilon}(t,x)-\rho(t,x)|^2+(\rho (t,\cdot)-\rho (t,x))^2*\varrho_\varepsilon\|_{L^{\frac{3}{2}}(0,T;L^{\frac{3}{2}}(\Omega^{'}))}\\
		&\ \ \left(m\|\nabla \varphi\|_{L^3(0,T;L^3(\Omega^{'}))}+m^2\|\nabla \rho^{\varepsilon}\|_{L^3(0,T;L^3(\Omega^{'}))}\right)\\
		&\leq C \varepsilon^{\frac{2}{3}}\|\rho\|_{L^3(0,T;{B^{\frac{1}{3}}_{3,c(\mathbb{N})}}(\Omega))}^2\left(m+m^2\varepsilon^{-\f{2}{3}}\|\rho\|_{L^3(0,T;{B^{\frac{1}{3}}_{3,c(\mathbb{N})}}(\Omega))}\right).
	\end{aligned}
\end{equation}
Then, taking into account the definition of Onsager's critical space \eqref{2.4}, one has
\begin{equation}\label{3.33}
\lim\limits_{m\rightarrow +\infty}\lim\limits_{\varepsilon\rightarrow 0} R_4=0.
\end{equation}
Collecting \eqref{3.21}, \eqref{3.23},   \eqref{3.25}, \eqref{3.27} and \eqref{3.33},  for any $\varphi(t,x)\in C_0^\infty((0,T)\times \Omega)$,  we have
\begin{equation}\label{3.34}
	\int_0^T \int_\Omega -v \partial_t \varphi(t,x)+(\omega \times v )\varphi(t,x)-\left( P(\rho)+\frac{1}{2}|v|^2\right) \text{div} \varphi(t,x)  dxdt=0,
\end{equation}
which means $(\rho,v)$ is also a weak solution of equations
\begin{equation}\label{3.35}\left\{\ba
	&\rho_t+\nabla \cdot (\rho v)=0, \\
	& v_t+\omega \times v+\nabla\left( P(\rho)+\frac{1}{2}|v|^2\right)=0.
	\ea\right.\end{equation}

	Then we have completed the proof of this proposition.
\end{proof}
Based on the above proposition, we are in a position to prove the main result.
\begin{proof}[Proof of Theorem \ref{the1}] To prove the helicity conservation of compressible Euler equations in a bounded domain, we will divide our proof into two steps. First, we will prove a local version helicity equality for compressible Euler equations, namely, for any smooth function $\phi(x)\in C_0^\infty (\Omega)$ with its compact support strictly contained in $\Omega$, i.e. $supp(\phi(x))=\Omega^{'}\subset\subset\Omega$, there holds
	\begin{equation}\label{localhelicity}
		\begin{aligned}
			&\int_{\Omega}(v\cdot\omega)(T,x)\phi(x)dx-\int_{\Omega}(v\cdot\omega)(0,x)\phi(x)dx\\
		-&\int_0^T\int_{\Omega}	\left(P(\rho )-\f12|v|^{2}\right)\omega\cdot \nabla \phi(x)+\left(\omega\cdot v\right)\B(v\cdot\nabla \phi(x)\B)dxdt=0.
\end{aligned}	\end{equation}
Then, by using some suitable cut-off function to capture the effect of the boundary, we extend the local helicity balance to be a global one.

\underline{Step 1}:
	Let $(\rho,v)$ be a weak solution to the compressible Euler equations \eqref{CEuler}, in combination with \eqref{a1}, \eqref{a4} and Proposition \ref{propo1}, we know that $(\rho,v)$ is also a weak solution to the system \eqref{3.35}. Then one can  mollify the equation \eqref{3.35} in spatial direction to deduce that
	\begin{equation}\label{3.36}\left\{\begin{aligned}
		&v^{\varepsilon}_{t} +(\omega\times v)^{\varepsilon}+\nabla(P(\rho )+\f12|v|^{2})^{\varepsilon}= 0,\\
		&\omega^{\varepsilon}_{t}-\text{curl}( v\times\omega)^{\varepsilon}=0,\\
		& \text{div}\omega^{\varepsilon} =0.
	\end{aligned}\right.\end{equation}
	Consequently, for any scalar function $\phi(x)\in C_0^\infty(\Omega)$ with its compact support strictly contained in $\Omega$, i.e. $supp(\phi(x))=\Omega^{'}\subset\subset\Omega$, by a straightforward computation, we have
	\begin{equation}\label{3.37}\ba
	&\f{d}{dt}\int_{\Omega^{'}}(v^{\varepsilon}\cdot\omega^{\varepsilon})\phi(x)dx
	=\int_{\Omega^{'}}\left( v_{t}^{\varepsilon}\cdot
	\omega^{\varepsilon}+\omega_{t}^{\varepsilon}\cdot v^{\varepsilon}\right)\phi(x)dx\\
	=&\int_{\Omega^{'}}-\left((\omega\times v)^{\varepsilon}+\nabla(P(\rho )+\f12|v|^{2})^{\varepsilon}\right)\cdot \omega^{\varepsilon}\phi(x)+
	\B(\text{curl}\left( v\times\omega\right)\B)^{\varepsilon}\cdot v^{\varepsilon}\phi(x)dx\\
	=&\int_{\Omega^{'}}-\left(\omega\times v\right)^{\varepsilon}\cdot \omega^{\varepsilon}\phi(x)+\B(P(\rho )+\f12|v|^{2}\B)^{\varepsilon}\text{div}\B(\omega^{\varepsilon} \phi(x)\B)+
	\left(   v\times\omega \right)^{\varepsilon}\cdot\text{curl}\B(v^{\varepsilon}\phi(x)\B)dx
	\\
	=&\int_{\Omega^{'}}-\left(\omega\times v\right)^{\varepsilon}\cdot \omega^{\varepsilon}\phi(x)+\left(P(\rho )+\f12|v|^{2}\right)^{\varepsilon}\omega^{\varepsilon}\cdot \nabla\phi(x)+
	\left(   v\times\omega \right)^{\varepsilon}\cdot\B(\omega^{\varepsilon}\phi(x)-\nabla\phi(x)\times v^{\varepsilon}\B)dx
	\\
	=&\int_{\Omega^{'}}-2\Big((\omega\times v)^{\varepsilon} -\omega^{\varepsilon}\times v^{\varepsilon} \Big)\cdot\omega^{\varepsilon}\phi(x)+\left(P(\rho )+\f12|v|^{2}\right)^{\varepsilon}\omega^{\varepsilon}\cdot \nabla\phi(x)\\
	&-\B( \left(  v\times\omega \right)^{\varepsilon}-v^{\varepsilon}\times \omega^{\varepsilon}\B)\cdot\B(\nabla\phi(x)\times v^{\varepsilon}\B)-\left(v^{\varepsilon}\times \omega^{\varepsilon}\right)\cdot\B(\nabla\phi(x)\times v^{\varepsilon}\B) dx,
	\ea\end{equation}
	where we have used the following identities
	 $$\omega^{\varepsilon}\times v^{\varepsilon}\cdot \omega^{\varepsilon}=\omega^{\varepsilon}\times\omega ^\varepsilon\cdot v^{\varepsilon}=0,\ \text{div}\omega^{\varepsilon}=0,$$
	 and
	 $$\text{curl}\B(v^{\varepsilon}\phi(x)\B)=\omega^{\varepsilon}\phi(x)-\nabla\phi(x)\times v^{\varepsilon}.$$
	Then integrating \eqref{3.37} over $[0,T]$, we have
	\begin{equation}\label{3.38}
		\ba
	&\int_{\Omega^{'}}(v^{\varepsilon}\cdot\omega^{\varepsilon})(T,x)\phi(x)dx-\int_{\Omega^{'}}(v^{\varepsilon}\cdot\omega^{\varepsilon})(x,0)\phi(x)dx\\
	=&\int_0^T\int_{\Omega^{'}}-2\Big((\omega\times v)^{\varepsilon} -\omega^{\varepsilon}\times v^{\varepsilon} \Big)\cdot\omega^{\varepsilon}\phi(x)+\left(P(\rho )+\f12|v|^{2}\right)^{\varepsilon}\omega^{\varepsilon}\cdot \nabla\phi(x)\\
	&-\left( \left(  v\times\omega \right)^{\varepsilon}-v^{\varepsilon}\times \omega^{\varepsilon}\right)\cdot\left(\nabla\phi(x)\times v^{\varepsilon}\right)-\left(v^{\varepsilon}\times \omega^{\varepsilon}\right)\cdot\left(\nabla\phi(x)\times v^{\varepsilon}\right) dxdt\\
	&=I_1+I_2+I_3+I_4.
		\ea
	\end{equation}
	Now it suffices to show that $I_1-I_4$ tends to zero as $\varepsilon\rightarrow0$.
	
First, by the H\"older inequality, one has
	\begin{equation}\label{3.39}\ba
	|I_1|\leq &C \|(\omega \times v)^{\varepsilon}-\omega^{\varepsilon}\times v^{\varepsilon}\|_{L^{\frac{3}{2}}(0,T;L^{\frac{3}{2}}(\Omega^{'}))}\|\omega ^\varepsilon\|_{L^3(0,T;L^3(\Omega^{'}))},
	\ea\end{equation}
which together with $v\in L^3(0,T;L^3(\Omega)),\ \omega\in L^3(0,T;L^3(\Omega))$ and Lemma \ref{lem2.2} gives
	$$
	\|(\omega \times v)^{\varepsilon}-\omega^{\varepsilon} \times v^{\varepsilon}\|_{L^{\frac{3}{2}}(0,T;L^{\frac{3}{2}}(\Omega^{'}))}\rightarrow 0,\ \ \text{as}\ \varepsilon\rightarrow 0.
	$$
	Hence as $\varepsilon\to 0$, there holds
	\begin{equation}\label{3.40}
		|I_1| \leq o(1).
	\end{equation}
	For $I_2$,  using the Lebesgue dominated convergence theorem, as $\varepsilon\to 0$, one has
	\begin{equation}\label{3.41}
		\ba
	I_2&=\int_0^T\int_{\Omega^{'}}	\left(P(\rho )+\f12|v|^{2}\right)^{\varepsilon}\omega^{\varepsilon}\cdot \nabla \phi(x)dxdt\to \int_0^T\int_{\Omega}	\left(P(\rho )+\f12|v|^{2}\right)\omega\cdot \nabla \phi(x)dxdt.
		\ea
	\end{equation}
Next, to control $I_3$, combining the H\"older inequality and Lemma \ref{lem2.2}, we have
\begin{equation}\label{3.44}
	\ba
	|I_3|=&|\int_0^T\int_{\Omega^{'}}\left( \B(  v\times\omega \B)^{\varepsilon}-v^{\varepsilon}\times \omega^{\varepsilon}\right)\cdot\B(\nabla\phi(x)\times v^{\varepsilon}\B)dxdt|\\
	\leq &C\|\B( \left(  v\times\omega \right)^{\varepsilon}-v^{\varepsilon}\times \omega^{\varepsilon}\B)\|_{L^{\f{3}{2}}(0,T;L^{\f{3}{2}}(\Omega^{'}))}\|v^{\varepsilon}\|_{L^3(0,T;L^3(\Omega^{'}))}\rightarrow 0,\ \text{as}\ \varepsilon\rightarrow 0.
	\ea
\end{equation}
Since for any vector field $\overrightarrow{u},\overrightarrow{v},\overrightarrow{w},\overrightarrow{x}$, the following identity holds
\begin{equation}\label{3.45}
	(\overrightarrow{u}\times \overrightarrow{v})\cdot(\overrightarrow{w}\times \overrightarrow{x})=(\overrightarrow{u}\cdot \overrightarrow{w})(\overrightarrow{v}\cdot \overrightarrow{x})-(\overrightarrow{u}\cdot \overrightarrow{x})(\overrightarrow{v}\cdot\overrightarrow{w}).
\end{equation}
Then, taking into account \eqref{3.45} and Lebesgue dominated convergence theorem, $I_4$ can be estimated as
\begin{equation}\label{3.46}
	\ba
	I_4&=\int_0^T\int_{\Omega^{'}}\left(v^{\varepsilon}\times \omega^{\varepsilon}\right)\cdot\B(\nabla\phi(x)\times v^{\varepsilon}\B) dxdt\\
	&=\int_0^T\int_{\Omega^{'}}\B(v^{\varepsilon}\cdot\nabla \phi(x)\B)\left(\omega^{\varepsilon}\cdot v^{\varepsilon}\right)-\left(v^{\varepsilon}\cdot v^{\varepsilon}\right)\B(\omega^{\varepsilon}\cdot \nabla \phi(x)\B)dxdt\\
	&\to \int_0^T\int_{\Omega}\B(v\cdot\nabla \phi(x)\B)\left(\omega\cdot v\right)-\left(v\cdot v\right)\B(\omega\cdot \nabla \phi(x)\B)dxdt, \ \text{as} \ \varepsilon\to0.
	\ea
\end{equation}
Then, combining the Lebesgue dominated convergence theorem once again  and \eqref{3.40}, \eqref{3.41}, \eqref{3.44} and \eqref{3.46}, we can get the local helicity conservation for compressible Euler equations, i.e. for any smooth function $\phi(x)\in C_0^\infty(\Omega)$, it holds that
\begin{equation}\label{3.47}\begin{aligned}	&\int_{\Omega}(v\cdot\omega)(T,x)\phi(x)dx-\int_{\Omega}(v\cdot\omega)(0,x)\phi(x)dx\\
		-&\int_0^T\int_{\Omega}	\left(P(\rho )-\f12|v|^{2}\right)\omega\cdot \nabla \phi(x)+\left(\omega\cdot v\right)\B(v\cdot\nabla \phi(x)\B)dxdt=0.\end{aligned}\end{equation}

\underline{Step 2}:   Next, we will extend the local helicity conservation to be a global one. Before we start the proof of this part, we recall the cut-off function used in
\cite{[BGSTW]} to capture  the affect of the boundary.
Let $\Omega$ be an open set with a bounded Lipschitz boundary $\partial\Omega$,
therefore the exterior normal vector to the boundary denoted by $\vec{n}(x)$ is defined almost everywhere.
 Then we observe that there  exists
a (small enough) $\varepsilon_{0}$ with the following properties: For $d(x,\partial\Omega)\leq  \varepsilon_{0} $, the function $x\rightarrow d(x,\partial\Omega)$
belongs to $W^{1,\infty}(\Omega)$ and  for almost every such $x$, there exists a unique point  $\sigma(x)\in\partial\Omega $ such that
$$d(x,\partial\Omega)<\varepsilon_{0}\Rightarrow d(x,\partial\Omega)=|x-\sigma(x)|\ \text{and}\ \nabla_{x}d(x,\partial\Omega)=-\vec{n}(\sigma(x)).$$
Before going further, we consider the smooth function $\psi(z)\in[0,1]$ with
$\psi(z)=0, z\in(-\infty,\f12]$  and $\psi(z)=1, z\in[1,\infty).$
Hence, for any $0<\varepsilon<\varepsilon_{0}$, we define the compactly supported function $\theta_{\varepsilon}(x)=\psi(\f{d(x,\partial\Omega)}{\varepsilon})$ which satisfies the following properties
\be\label{testproperties}
\left\{\ba
&\theta_{\varepsilon}(x)=0, \ \text{if}\ x\in \Omega_{\f{\varepsilon}{2}} \ \text{and}\
\theta_{\varepsilon}(x)=1, \ \text{if}\ x\in \Omega^{\varepsilon},
\\
&  \nabla\theta_{\varepsilon}(x)=-\f{1}{\varepsilon}\psi'(\f{d(x,\partial\Omega)}{\varepsilon})\vec{n}(\sigma(x))\neq0 ,\ \text{if}\  \f{\varepsilon}{2}<d(x,\partial\Omega)<\varepsilon. \ea\right.
\ee
Then let  $\phi(x)=\theta_\varepsilon(x)$ in \eqref{3.47}, we have
\begin{equation}\label{3.48}\begin{aligned}
		&\int_{\Omega}(v\cdot\omega)(T,x)\theta_\varepsilon(x)dx-\int_{\Omega}(v\cdot\omega)(0,x)\theta_\varepsilon(x)dx\\
		-&\int_0^T\int_{\Omega}	\left(P(\rho )-\f12|v|^{2}\right)\omega\cdot \nabla \theta_\varepsilon(x)+\left(\omega\cdot v\right)\B(v\cdot\nabla \theta_\varepsilon(x)\B)dxdt=0.\end{aligned}\end{equation}
	For the third term of above equation, due to the continuous condition \eqref{a2}, we know that
	\begin{equation}\label{3.49}
		\begin{aligned}
		&	|\int_0^T\int_{\Omega}	\left(P(\rho )-\f12|v|^{2}\right)\omega\cdot \nabla \theta_\varepsilon(x)+\left(\omega\cdot v\right)\B(v\cdot\nabla \theta_\varepsilon(x)\B)dxdt|\\
			\leq &C\int_0^T \f{1}{\varepsilon}\int_{\f{\varepsilon}{2}<d(x,\partial \Omega)<\varepsilon}	\B|\left(P(\rho )-\f12|v|^{2}\right)\omega\cdot  \vec{n}(\sigma(x))+\left(\omega\cdot v\right)v\cdot \vec{n}(\sigma(x))\B|dxdt
			\to  0.
		\end{aligned}
	\end{equation}
Combining \eqref{3.48}, \eqref{3.49} and the Lebesgue dominated convergence theorem, we have
\begin{equation}\label{3.50}
	\int_{\Omega}(v\cdot\omega)(T,x)dx=\int_{\Omega}(v\cdot\omega)(0,x)dx=0.
	\end{equation}
 Thus the proof of Theorem \ref{the1} is completed.
\end{proof}

{\bf Conflict of Interest Statement}

This work does not have any conflicts of interest.

{\bf Data Availability}

This publication is supported by multiple datasets, which are openly available at locations cited in the reference section.

\end{document}